\documentclass{article}
\usepackage{graphicx}

\usepackage{comment}
\usepackage{amsmath,amsthm,amsfonts,amssymb,amscd}
\usepackage{authblk}
\usepackage{geometry}
\usepackage{mathtools}
\usepackage{float}
\usepackage{graphicx}
\usepackage{enumitem}
\usepackage{tikz}
\usepackage[square,sort,comma,numbers]{natbib}
\usepackage{subcaption}
\usepackage{hyperref}
\usepackage{thmtools} 
\usepackage{thm-restate}

\usepackage{mathrsfs}
\hypersetup{
     colorlinks=true,
     linkcolor=blue,
     filecolor=blue,
     citecolor = black,      
     urlcolor=cyan,
     }
\usepackage[T1]{fontenc}
\newtheorem{remark}{Remark}[section]
\newtheorem{theorem}{Theorem}[section]
\newtheorem{proposition}{Proposition}[section]
\newtheorem{corollary}{Corollary}[theorem]
\newtheorem{lemma}{Lemma}[section]

\theoremstyle{definition}

\newtheorem{example}{Example}[section]
\newtheorem{assumption}{Assumption}[section]

\newcommand{\R}{\mathbb{R}}

\newcommand{\Z}{\mathbb{Z}}

\newcommand{\E}{\mathbb{E}}

\newcommand{\vc}{\vcentcolon}
\newcommand{\eps}{\varepsilon}
\newcommand{\kk}{k}
\newcommand{\ind}{\mathbf{1}}
\newcommand{\ff}{\mathsf{f}^{\eps}}
\DeclareMathOperator*{\argmin}{arg\,min}
\newcommand{\vv}{\mathsf{v}}
\newcommand{\ww}{\mathsf{w}}
\DeclareMathOperator{\Ima}{Im}

\def\X{{\mathcal X}}

\def\M{{\mathcal M}}
\newcommand{\grad}{\mathrm{grad}}

\usepackage{amsmath}

\usepackage[capitalise]{cleveref}
\usepackage{todonotes}
\newcommand{\dive}{\mathrm{div}}

\title{Convergence of the Sinkhorn Riemannian metric for finitely supported measures}
\author{Gilles Mordant\thanks{Department of Applied and Computational Mathematics, Yale University} \ and Liane Xu\thanks{Program in Applied and Computational Mathematics, Princeton University}}
\date{\vspace{-7.5ex}}
\begin{document}
\maketitle
\begin{abstract}
  The Sinkhorn Riemannian metric characterizes the local behavior of the Sinkhorn divergence, a debiased version of entropy-regularized optimal transport. Although the Sinkhorn divergence is known to approximate the squared Wasserstein-2 distance as the regularization parameter $\eps$ goes to $0$, the behavior of the Sinkhorn Riemannian metric as $\eps \to 0$ remains largely open. To the best of our knowledge, the only existing argument is for measures with a density and is still a formal computation. In this work, we prove the convergence of the Sinkhorn Riemannian metric as $\eps \to 0$ in the case of finitely supported measures undergoing horizontal perturbations. Under our assumptions, our techniques additionally allow us to prove the convergence of higher order derivatives of the Sinkhorn Riemannian metric, thereby ensuring the convergence of the associated Riemann curvature tensor and Christoffel symbols as $\eps \to 0$. We also compare with the absolutely continuous setting and show where our approach fails for measures with a density.
\end{abstract}

 \let\thefootnote\relax


\section{Introduction}

Entropy-regularized optimal transport has become popular in machine learning and data science for approximating Wasserstein distances, thanks to its computational \cite{cuturi2013sinkhorn, peyre2019computational} and statistical benefits \cite{genevay2019sample, mena2019statistical,groppe2024lower,rigollet2025sample}. For $\X \subset \R^d$ compact, $\mu, \nu \in \mathcal{P}(\X)$ and $\eps > 0$, entropy-regularized optimal transport is given by
\begin{equation}
\label{eq:ent-reg-ot}
    \mathrm{OT}_\eps(\mu, \nu) \vc= \inf_{\pi \in \Pi(\mu, \nu)} \int \|x-y\|^2 d\pi(x, y) + \eps \mathrm{KL}(\pi | \mu \otimes \nu),
\end{equation}
where $\Pi(\mu, \nu)$ is the set of couplings between $\mu$ and $\nu$ and $\mathrm{KL}$ is the Kullback–Leibler divergence.
As $\eps \to 0$, $\mathrm{OT}_\eps(\mu, \nu)$ converges to the squared Wasserstein-2 distance
\[ W_2(\mu,\nu)^2 \vc = \inf_{\pi \in \Pi(\mu, \nu)} \int \|x-y\|^2 d\pi(x, y) \]
between $\mu$ and $\nu$ \cite{nutz2022introduction}. 
$\mathrm{OT}_\eps(\mu, \mu)$ is often strictly positive (see \eqref{eq: SinkhLB}), which is why Genevay et al. proposed using the \textit{Sinkhorn divergence}
\begin{equation}
\label{eq:def-sinkdiv}
    S_\eps(\mu, \nu) \vc = \mathrm{OT}_\eps(\mu, \nu) - \frac{\mathrm{OT}_\eps(\mu, \mu)}{2} - \frac{\mathrm{OT}_\eps(\nu, \nu)}{2}
\end{equation}
as a loss function \cite{genevay2018learning}. Feydy et al. later showed that $S_\eps \geq 0$ and $S_\eps(\mu, \nu) = 0$ if and only if $\mu = \nu$ \cite{feydy2019interpolating}.

In parallel, the formal Riemannian structure of Wasserstein-2 space \cite{otto2001geometry} has also become useful in practice, on top of having a rich mathematical theory; for example, see \cite[Chapter 6]{chewi2024statisticaloptimaltransport} for applications of Wasserstein gradient flows and \cite{ambrosio2008gradient} for a more theoretical treatment. This Riemannian structure is given by defining the tangent space at $\mu \in \mathcal{P}_2(\R^d)$ by
\begin{equation}
\label{eq:wasserstein-tangent-space}
T_\mu \mathcal{P}_2(\R^d) \vc = \overline{ \{ \nabla \varphi \colon \varphi \in C_c^\infty(\R^d) \}}^{L^2(\mu)} 
\end{equation}
and equipping $T_\mu \mathcal{P}_2(\R^d)$ with the $L^2(\mu)$ inner product. This Riemannian structure can be motivated via the Benamou-Brenier formula
\[ W_2(\mu_0, \mu_1)^2 = \inf_{(\mu_t, v_t)} \left\{ \int_0^1 \|v_t\|^2_{L^2(\mu_t)} dt \colon \partial_t \mu_t + \dive(v_t\mu_t) = 0 \right\}, \]
where the infimum is over sufficiently regular curves $t \mapsto \mu_t$ and the continuity equation is understood in the distributional sense \cite{benamou2000computational}.
The definition of the tangent space (\ref{eq:wasserstein-tangent-space}) can then be understood by observing that $v \in T_\mu \mathcal{P}_2(\R^d)$ if and only if
\begin{equation}
\label{eq:minimal-l2}
    \left\| v+ w \right\|_{L^2(\mu)} \geq \|v\|_{L^2(\mu)}
\end{equation}
for all vector fields $w \in L^2(\mu)$ such that $\dive(w\mu) = 0$ \cite[Lemma 8.4.2]{ambrosio2008gradient}.
Although the Wasserstein-2 Riemannian structure is often defined via vector fields in the tangent spaces (\ref{eq:wasserstein-tangent-space}), if $\dot{\mu} = -\dive(v \mu)$ for $v \in T_\mu \mathcal{P}_2(\R^d)$, we write
\[ g^0_\mu(\dot{\mu}, \dot{\mu}) = \|v\|_{L^2(\mu)}^2 = \|\dot\mu\|_{\dot{H}^{-1}(\mu)}^2, \]
where $\| \cdot\|_{\dot{H}^{-1}(\mu)}$ is a homogeneous negative Sobolev norm (see \cite{peyre2018comparison}).
If $\mu$ has a density and we are only given $\dot{\mu}$ where 
\begin{equation}
\label{eq:pde-for-riem-metric}
\dot{\mu} = -\dive(\mu \nabla \varphi)
\end{equation}
for some $\varphi$, with appropriate conditions, one can also obtain $g^0_\mu(\dot{\mu}, \dot{\mu})$ by first solving the PDE (\ref{eq:pde-for-riem-metric}) for $\varphi$.

These two facets of optimal transport are complementary: regularization has imposed itself as a convenient numerical tool to approximate optimal transport, while the geometric structure of Wasserstein-2 space is theoretically fundamental.
Combining these two aspects, \cite{lavenant2025riemannian} recently studied the Riemannian structure arising Sinkhorn divergences,
by considering paths $\{\mu_t\}_{t \in (-\delta, \delta)}$ satisfying appropriate regularity assumptions and taking
\begin{equation}
    g^\eps_\mu(\dot{\mu}, \dot{\mu}) \vc = \lim_{t \to 0} \frac{S_\eps(\mu, \mu_t)}{t^2},
\end{equation}
where $\dot{\mu}$ denotes the derivative of $\{\mu_t\}_{t \in (-\delta, \delta)}$ at $t = 0$.
We will refer to $g^\eps$ as the \textit{Sinkhorn Riemannian metric}. 

The Sinkhorn Riemannian metric has already appeared in several applications.
In earlier work, \cite{shen2020sinkhorn} used it in a gradient-based algorithm to optimize a functional on a parametric family of probability measures; however, the focus was on training generative models and the Riemannian structure itself was studied in less detail compared to \cite{lavenant2025riemannian}. In \cite{hardion2025gradient}, the Sinkhorn Riemannian metric appears when studying the JKO scheme with the squared Wasserstein-2 distance replaced by the Sinkhorn divergence. The Sinkhorn Riemannian metric also plays a central role in \cite{agarwal2026sinkhorn}, where they define treatment effects using the Sinkhorn divergence, and in \cite{kokot2025coreset}, where they propose an algorithm for efficiently compressing a dataset with the Sinkhorn divergence as the measure of error. Moreover, given the recent work on manifold learning in Wasserstein space \cite{hamm2025manifold, oliver2025laplace}, one can also study the problem of manifold learning in the geometry of the Sinkhorn Riemannian metric \cite{xu2026convergence}.

Since the Sinkhorn divergence approximates the squared Wasserstein-2 distance as $\eps\to 0$, a natural question is whether the Sinkhorn Riemannian metric also converges to the Wasserstein-2 Riemannian structure as $\eps \to 0$. 
An informal argument for measures with density is provided in \cite{lavenant2025riemannian}, but to the best of our knowledge, this question has not been studied for finitely supported measures.

In this paper, we prove the convergence of the Sinkhorn Riemannian metric (and its derivatives) as $\eps \to 0$ for measures supported on $N$ points which vary smoothly. Specifically, we will work under the following assumption:
\begin{restatable}[]{assumption}{assump}
\label{assump:1}
Fix $a_1, \ldots, a_N > 0$ such that $\sum_k a_k = 1$, a compact set $\X \subset \R^d$, an open set $U \subset \R^m$ and a smooth embedding $\iota \colon U \to \X^N$ such that $\iota_k(p) \neq \iota_l(p)$ for all $p \in U$ and $k \neq l$. Define for each $p  \in U$
\[ \mu_p \vc = \sum_{k = 1}^N a_k \delta_{\iota_k(p)}. \]
\end{restatable}


The set $U$ can be replaced with a more general smooth manifold $\M$ by passing to charts. By a slight abuse of notation, we treat $g^\eps$ and $g^0$ as Riemannian metrics directly on $U$, e.g., for a smooth path $\gamma \colon (-\delta, \delta) \to U$ with $\gamma(0) = p$,
\[ 
g^\eps_p(\gamma'(0), \gamma'(0)) = \lim_{t \to 0} \frac{S_\eps(\mu_{\gamma(0)}, \mu_{\gamma(t)})}{t^2}
\]
and similarly for $g^0$. Working in the standard coordinates on $U \subset \R^m$, we show that
\[ \lim_{\eps \to 0} g_{ij}^\eps(p) = g_{ij}^0(p) \]
for each $p \in U$ and $i, j \in [N]$; additionally, this convergence can be made uniform over compact subsets of $U$.

Under Assumption \ref{assump:1}, each $g^\eps$ and $g^0$ turn out to be smooth Riemannian metrics on $U$, so another natural question is whether the derivatives of $g^\eps_{ij}$ converge as $\eps \to 0$ at each $p \in U$. 
We answer this in the affirmative as well. Again, the convergence can be made uniform over compact subsets of $U$. The convergence of the derivatives of $g^\eps$ at $p \in U$ as $\eps \to 0$ then implies the convergence of the associated Christoffel symbols and Riemann curvature tensor at $p$ as $\eps \to 0$.

The rest of the paper is organized as follows. In \cref{sec:background}, we first introduce some background about Sinkhorn divergences and operators on graphs, along with the notation that we will be working with. We then prove the convergence of $g^\eps$ as $\eps \to 0$ in \cref{sec:conv}, and the convergence of its derivatives in \cref{sec:derivs}.

\section{Preliminaries}
\label{sec:background}

\subsection{Sinkhorn divergences}
\label{subsec:sinkhorn}
Here we give a brief overview of entropy-regularized optimal transport and Sinkhorn divergences; see \cite{peyre2019computational, nutz2022introduction, chewi2024statisticaloptimaltransport} for more details. Let $\X \subset \R^d$ be compact, $\mu, \nu \in \mathcal{P}(\X)$ and $\eps > 0$. Recall from the introduction that
\[
    \mathrm{OT}_\eps(\mu, \nu) \vc= \inf_{\pi \in \Pi(\mu, \nu)} \int \|x-y\|^2 d\pi(x, y) + \eps \mathrm{KL}(\pi | \mu \otimes \nu),
\]
where $\Pi(\mu, \nu)$ denotes the set of couplings between $\mu$ and $\nu$ and $\mathrm{KL}$ denotes the Kullback–Leibler divergence.
For us, the dual formulation
\begin{equation}
    \label{eq:ent-reg-ot-dual}
    \mathrm{OT}_\eps(\mu, \nu) = \sup_{f, g \in C(\X)} \int f d\mu + \int g d\nu - \eps \int \left( \exp\left( \frac{f(x) + g(y) - \|x-y\|^2}{\eps}\right)-1 \right)d(\mu \otimes \nu)(x, y)
\end{equation} 
will be important.
In particular, there exist maximizers $f_{\mu, \nu}^\eps, g_{\mu, \nu}^\eps \in C(\X)$ for (\ref{eq:ent-reg-ot-dual}) which satisfy 
\begin{align}
    \label{eq:schrodinger-1}
    \begin{split}
    \int \exp \left(\frac{f_{\mu, \nu}^\eps(x) + g_{\mu, \nu}^\eps(y) - \|x-y\|^2}{\eps} \right) d\nu(y) &= 1 \\
    \int \exp \left(\frac{f_{\mu, \nu}^\eps(x) + g_{\mu, \nu}^\eps(y) - \|x-y\|^2}{\eps} \right) d\mu(x) &= 1, 
    \end{split}
\end{align}
or equivalently,
\begin{align}
\label{eq:schrodinger-1-1}
\begin{split}
f_{\mu, \nu}^\eps(x) &= -\eps \log \left( \int \exp \left( \frac{g_{\mu, \nu}^\eps(y) - \|x-y\|^2}{\eps} \right) d\nu(y) \right) \\
g_{\mu, \nu}^\eps(y) &= -\eps \log \left( \int \exp \left( \frac{f_{\mu, \nu}^\eps(x) - \|x-y\|^2}{\eps} \right) d\mu(x) \right) 
\end{split}
\end{align}
for all $x, y \in \X$. Such $f_{\mu, \nu}^\eps, g_{\mu, \nu}^\eps$ are known as \textit{Schr\"odinger potentials} and are unique up to an additive constant. 
Moreover, the coupling
\begin{equation}
\label{eq:optimal-coupling-form}
    d\pi_{\mu, \nu}^\eps(x, y) \vc= \exp \left( \frac{f_{\mu, \nu}^\eps(x) + g_{\mu,\nu}^\eps(y) - \|x-y\|^2}{\eps} \right)d(\mu \otimes \nu)(x, y)
\end{equation}
is optimal for the primal (\ref{eq:ent-reg-ot}).

In the case where $\mu = \nu$, we can choose $f_{\mu, \mu}^\eps, g_{\mu, \mu}^\eps$ so that $f_{\mu, \mu}^\eps = g_{\mu, \mu}^\eps.$ This is the representative that we will use throughout this paper. Under this convention, the \textit{self-transport kernel} at $\mu$ is given by
\begin{equation}
    \label{eq:self-transport-kernel}
    \kk_\mu^\eps(x, y) \vc = \exp \left( \frac{f_{\mu, \mu}^\eps(x) + f_{\mu,\mu}^\eps(y) - \|x-y\|^2}{\eps} \right).
\end{equation}
Since the Gaussian kernel is a positive definite kernel, $\kk_\mu^\eps$ is as well. Therefore, by the Moore-Aronzajn theorem \cite{aronszajn1950theory}, there exists a unique reproducing kernel Hilbert space (RKHS) with kernel $\kk_\mu^\eps$, which we denote by $\mathcal{H}_\mu^\eps$. 
The dual formulation also gives us the lower bound
\begin{equation}
\label{eq: SinkhLB}
   \mathrm{OT}_\eps(\mu, \mu)\geq  \eps \left( 1 - \int \exp\left( - \frac{\|x-y\|^2}{\eps} \right) d(\mu \otimes \mu)(x,y)  \right).
\end{equation}

\subsubsection{The Sinkhorn Riemannian metric}
We now discuss the Riemannian structure induced by Sinkhorn divergences, following \cite{lavenant2025riemannian}. Fix any $k \in \Z_{\geq 0}$ and $\mu \in \mathcal{P}(\X)$. If $k \geq 1$, furthermore assume that $\X$ is the closure of a bounded open set. 
Define the operators 
\[K_\mu^\eps \colon C(\X) \to C^k(\X), \quad K_\mu^\eps[\varphi](y) \vc =  \int_{\X} \varphi(x) \kk_\mu^\eps(x, y) d\mu(x) \] 
\[H_\mu^\eps \colon (\mathcal{H}_\mu^\eps)^* \to \mathcal{H}_\mu^\eps, \quad H_\mu^\eps[\nu](y) \vc = \langle \nu, \kk_\mu^\eps(\cdot, y) \rangle. \]
The definition of the operator $K_\mu^\eps$ can also be extended to continuous vector-valued functions $\varphi \colon \X \to \R^d$ by evaluating component-wise. Of particular interest is when $\varphi(x) = x$ for $x \in \X$, which we denote by
\begin{equation}
\label{eq:entropic-map}
T^\eps_{\mu}(y) \vc = \int_{\X} x \kk_\mu^\eps(x, y) d\mu(x). 
\end{equation}
This is the \emph{entropic map} (e.g., \cite{chewi2024statisticaloptimaltransport, pooladian2021entropic}) for the self-transport problem at $\mu$.

The operator $H_\mu^\eps$ is precisely the kernel distribution embedding \cite{simon2018kernel} for $\mathcal{H}_\mu^\eps$, a generalization of the popular kernel mean embedding.
By (\ref{eq:schrodinger-1-1}), $f_{\mu, \mu}^\eps$ and $\kk_\mu^\eps$ can be extended to smooth functions on $\R^d$ and $\R^d \times \R^d$, respectively, so $\mathcal{H}_\mu^\eps$ injects continuously into $C^k(\X)$ \cite{lavenant2025riemannian, simon2018kernel}. Therefore, the operator $H_\mu^\eps$ is also well-defined on $C^k(\X)^*$.

To construct the Sinkhorn Riemannian metric, consider a path $\{\mu_t\}_{t \in (-\delta, \delta)}$ with $\mu_0 = \mu$ which is continuously differentiable in $C^k(\X)^*$ with the weak-* topology. Let $C^k(\X)/\R$ denote the Banach space obtained by quotienting out the constant functions from $C^k(\X)$, equipped with the standard quotient norm. Then
\begin{equation}
\label{eq:hess-sink}
    g^\eps_\mu(\dot{\mu}, \dot{\mu}) \vc = \lim_{t \to 0} \frac{S_\eps(\mu_0, \mu_t)}{t^2} = \frac{\eps}{2} \langle \dot{\mu}, (\mathrm{id} - (K_\mu^\eps)^2)^{-1} H_\mu^\eps [\dot{\mu}] \rangle,
\end{equation}
where $\dot{\mu}$ is the derivative of $\{\mu_t\}_{t \in (-\delta, \delta)}$ at $t = 0$ \cite{lavenant2025riemannian}. Here, $\mathrm{id} - (K_\mu^\eps)^2$ should be understood as an operator from $C^k(\X) / \R$ to itself; \cite{lavenant2025riemannian} shows that this is well-defined and has a bounded inverse. Since $\dot{\mu}$ is well-defined as an operator on the space $C^k(\X)/\R$, the right hand side of (\ref{eq:hess-sink}) is indeed well-defined.

\subsection{Operators on graphs}
\label{subsec:graph-op}
The main difficulty in controlling $g^\eps$ is understanding how the inverse of $\mathrm{id} - (K_\mu^\eps)^2$ behaves as $\eps \to 0$. 
When $\mu = \sum_{k = 1}^N a_k \delta_{x_k}$ is a discrete measure with finite support, $K_\mu^\eps$ can be expressed as left multiplication by the $N \times N$ matrix 
\[ 
    K \vc = \left[ a_l \kk_\mu^\eps(x_k, x_l)  \right]_{k, l}, 
\]
which is a Markov transition matrix by (\ref{eq:schrodinger-1}) (see also \cite[Remark 3.10]{lavenant2025riemannian}). Since $K_\mu^\eps$ has non-negative spectrum \cite{lavenant2025riemannian}, the inverse of $I + K_\mu^\eps$ can be controlled, and the difficulty lies more in (pseudo-) inverting $I - K$, which can be viewed as a graph Laplacian.

For this reason, we review the analogues of the gradient, divergence and Laplacian for a weighted graph, following \cite{lim2020hodge}.
Let $G = (V, E, W)$ be a weighted graph with $|V|$ finite and symmetric non-negative edge weights, i.e., $W_{kl} = W_{lk} \geq 0$ for all $k, l$. The \textit{degree} of vertex $k$ is given by
\[
d_k \vc = \sum_l W_{kl}.
\]
Denote by $L^2(V)$ the set of real-valued functions $f \colon V \to \R$, equipped with the inner product
\begin{equation}
\label{eq:l2v}
\langle f, g \rangle_{L^2(V)} \vc = \sum_k d_k f(k) g(k).
\end{equation}
Define the set of \textit{edge flows} on $G$ as
\[ 
L^2_{\wedge}(E) \vc = \{ X \colon V \times V \to \R \ |\ X(k, l) = -X(l, k) \ \forall k, l \}, 
\]
equipped with the inner product
\begin{equation}
    \langle X, Y \rangle_{L^2_{\wedge}(E)} \vc = \sum_{k < l} W_{kl} X(k, l) Y(k, l) = \frac{1}{2} \sum_k \sum_l W_{kl} X(k, l) Y(k, l).
\end{equation}
We will often treat $f \colon V \to \R$ and $X \colon V \times V \to \R$ as elements of $\R^{|V|}$ and $\R^{|V| \times |V|}$, respectively.
The \textit{gradient} $\grad_G \colon L^2(V) \to L^2_{\wedge}(E)$ and \textit{divergence} $\dive_G \colon L^2_{\wedge}(E) \to L^2(V)$ on $G$ are defined as
\[(\grad_G f)(k, l) \vc =  f(l) - f(k), \ \ \  (\dive_G X)(k) \vc = \frac{1}{d_k} \sum_l W_{kl} X(k, l), \]
respectively.
One can check that the divergence $\dive_G$ is the negative of the adjoint of $\grad_G$:
\begin{equation}
\label{eq:div_adjoint}
\langle \grad_G f, X \rangle_{L^2_{\wedge}(E)} = -\langle f, \dive_G X \rangle_{L^2(V)}
\end{equation}
for all $f \in L^2(V)$ and $X \in L^2_{\wedge}(E)$; this can also be viewed as an integration by parts formula on $G$.
The \textit{graph Laplacian} $L_G$ on $G$ is defined similarly to the smooth setting:
\[
    L_Gf \vc = -\dive_G(\grad_G f)
\]
for $f \in L^2(V)$. This agrees with the typical definition of the graph Laplacian (with the random walk normalization)
\begin{equation}
\label{eq:gl-rw}
L_G = I - D_G^{-1}W, 
\end{equation}
where $D_G$ is the degree matrix of $G$ (a diagonal matrix with entries $d_1, \ldots, d_{|V|}$ on the diagonal).

\subsection{Notation and assumptions}
\label{subsec:notation}

As explained in the introduction, unless otherwise specified, we work under the following assumption:
\assump*

Without loss of generality, we can take $\X$ to be the closure of a bounded open subset of $\R^d$. For $n \in \Z_{> 0}$, let $[n]:=\{1, \ldots, n\}$. From now on, we use $i, j \in [m]$ to index the coordinates on $U$ and $k, l \in [N]$ to index the particles. We treat $a_1, \ldots, a_N > 0$ and $N \in \Z_{> 0}$ as constants when using big-O notation $O(\cdot)$. If $\mathcal{F}$ is a Banach space containing the constant function $\ind_\mathcal{X}(x) \vc = 1$ on $\mathcal{X}$, we use $\mathcal{F}/\R$ to denote the Banach space from quotienting out $\mathrm{span}(\left\{\ind_\mathcal{X}\right\})$.
We use the shorthand notation $\langle \cdot, \cdot\rangle_a$ for the $a$-weighted inner product, i.e., 
\[\langle u, v\rangle_a \vc= \sum_k a_k u_k v_k.\]
Unless otherwise specified, $A^\dagger$ denotes the generalized inverse of $A \in \R^{N \times N}$ as an operator from $\R^N$ \textit{with the $a$-weighted inner product $\langle \cdot, \cdot \rangle_a$} to itself. Specifically, when $A$ is self-adjoint with respect to $\langle \cdot, \cdot \rangle_a$, $A$ has an orthonormal (with respect to $\langle \cdot, \cdot \rangle_a$) basis of eigenvectors $u^{(1)}, \ldots, u^{(N)}$ with corresponding eigenvalues $\lambda_1, \ldots, \lambda_N$, and $A^\dagger$ is given by
\begin{equation}
\label{eq:generalized-inverse}
A^\dagger u^{(k)} =  
\begin{cases}
    \frac{1}{\lambda_k} u^{(k)} & \text{ if }\lambda_k \neq 0  \\
    0 & \text{ otherwise}.
\end{cases}
\end{equation}
 

\subsubsection{Velocities}
We denote the velocities of the atoms by
\begin{align*}
    v_{k, i} \colon U &\to \R^d \\
    p &\mapsto \frac{\partial \iota_k}{\partial p_i}(p)
\end{align*}
for each $i \in [m]$ and $k \in [N]$.
Relatedly, for each $i \in [m]$, define $\dot{\mu}_i \colon U \to C^1(X)^*$ by
\[\quad \langle \dot{\mu}_i(p), \varphi \rangle \vc= \lim_{t \to 0} \bigg\langle \frac{\mu_{p + te_i} - \mu_p}{t}, \varphi \bigg\rangle = \sum_{k = 1}^N a_k \nabla \varphi (\iota_k(p)) \cdot v_{k, i}(p) \]
for $p \in U$ and $\varphi \in C^1(\X),$ where $\{e_i\}_{i \in [m]}$ denotes the standard coordinate basis of $\R^m$.

\subsubsection{Riemannian metrics}
We will often work in coordinates: for a Riemannian metric $g$ on $U$, 
\[ g_{ij}(p) \vc = g_p \left( \frac{\partial}{\partial p_i} \bigg|_p,  \frac{\partial}{\partial p_j}\bigg|_p \right) \]
for $p \in U$.
By a slight abuse of notation, we consider the Wasserstein-2 Riemannian metric $g^0$ and the Sinkhorn Riemannian metric $g^\eps$ for each $\eps > 0$ as a Riemannian metric directly on $U$. 
For $g^0$, this means
\[ g^0_{ij} = \sum_k a_k v_{k, i}^\top v_{k, j} = \langle v_{k, i}, v_{k, j}\rangle_a. \]
For a smooth path $\gamma \colon (-\delta, \delta) \to U$ with $\gamma(0) = p$, it can be checked that $\{\mu_{\gamma(t)}\}_{t \in (-\delta, \delta)}$ is continuously differentiable in $C^1(\X)^*$ with the weak-* topology, so
\[ 
g^\eps_p(\gamma'(0), \gamma'(0)) = \lim_{t \to 0} \frac{S_\eps(\mu_{\gamma(t)}, \mu_{\gamma(0)})}{t^2}. 
\]
Alternatively, for each $\eps > 0$, the function $F_\eps(p, q) \vc = S_\eps(\mu_p, \mu_q)$ is smooth on $U \times U$ by an implicit function theorem argument (see \cite{xu2026convergence} for a proof specifically for Assumption \ref{assump:1}, though similar arguments have appeared previously in the literature, e.g., \cite{lavenant2025riemannian, luise2018differential, carlier2024displacement}). Therefore we also have
\[ g_p^\eps = \frac{1}{2} \nabla^2 [F_\eps(p, \cdot)](p), \]
where $\nabla^2$ denotes the Hessian.
\subsubsection{Self-transport and related RKHSs}
Instead of $\kk_{\mu_p}^\eps, \mathcal{H}_{\mu_p}^\eps, H_{\mu_p}^\eps$ (see \cref{subsec:sinkhorn}), we will write $\kk_p^\eps, \mathcal{H}_p^\eps, H_p^\eps$, respectively. 
We treat
\[ \pi^\eps \colon U \to \R^{N \times N}, \ \ K^\eps \colon U \to \R^{N \times N} \]
as matrix-valued functions, where $\pi^\eps(p)$ and $K^\eps(p)$ correspond to $\pi_{\mu_p, \mu_p}^\eps$ and $K^\eps_{\mu_p}$:
\[ \pi^\eps(p) = \bigg[ a_k a_l \kk_p^\eps(\iota_k(p), \iota_l(p)) \bigg]_{k, l}, \quad K^\eps (p) = \bigg[ a_l \kk_p^\eps(\iota_k(p), \iota_l(p)) \bigg]_{k, l} \]
for $p \in U$. 
Similarly, define
\begin{align*}
\ff \colon U &\to \R^N \\
p &\mapsto \bigg[ f^\eps_{\mu_p, \mu_p}(\iota_k(p)) \bigg]_k.
\end{align*}
We denote the off-diagonal mass of $\pi^\eps(p)$ by
\[ \alpha^\eps(p)  \vc = \sum_{k \neq l} \pi_{kl}^\eps(p) \]
for each $p \in U$.
 We will also consider vector-valued functions obtained from $p \mapsto H_p^\eps[\dot{\mu}_i(p)]$ for each $i \in [m]$:
\begin{equation}
\label{eq:def-h-i}
h_i^\eps \colon U \to \R^N, \quad h_i^\eps(p) \vc = \bigg[ H_p^\eps[\dot{\mu}_i(p)](\iota_k(p)) \bigg]_k. 
\end{equation}
Furthermore, 
define $m^\eps \colon U \to \R^N$ by
\begin{equation}
\label{eq:m-eps}
       m^\eps_k(p) \vc = T^\eps_{\mu_p}(\iota_k(p)) = \sum_{l} K^\eps_{kl}(p) \iota_l(p)
\end{equation}
for each $k \in [N]$ and $p \in U$. 

\subsubsection{Other quantities depending on $p$}

For each $p\in U$, define
\[
 B(p) \vc = \max_{i, k} \| v_{k, i}(p)\|, \ \ c(p) \vc = \min_{k \neq l} \|\iota_k(p) - \iota_l(p)\|^2,  \ \ D(p) \vc = \max_{k, l} \|\iota_k(p) - \iota_l(p)\|. 
\]
These appear in the quantitative bound in \cref{thm:conv-riem-metric}, which implies the convergence of $g^\eps$ to $g^0$ uniformly on compact subsets of $U$.
However, we will occasionally use big-O notation at a fixed $p \in U$, in which case we treat quantities depending on $\iota$ and its derivatives at $p$ as constants and focus on the behavior as $\eps \to 0$.

\section{Convergence of the Sinkhorn Riemannian metrics}
\label{sec:conv}

In this section, we study the convergence of $g^{\eps}$ as $\eps \to 0$. Our main result is the following.

\begin{theorem}[Convergence of $g^\eps$ as $\eps \to 0$]
\label{thm:conv-riem-metric}
For any $i, j \in [m]$, we have
\begin{equation}
\label{eq:quantitative-bound}
\left| g_{ij}^\eps - g_{ij}^0 \right| \leq 2 \alpha^\eps B^2 + \frac{6\alpha^\eps(BD)^2}{\eps}.
\end{equation}
In particular, since $\alpha^\eps \leq N e^{-c/\eps}$,
\begin{equation}
\label{eq:conv-riem-metric}
\lim_{\eps \to 0} g_{ij}^\eps(p) = g_{ij}^0(p)
\end{equation}
for each $p \in U$.
\end{theorem}

Before proceeding, we note that a result from \cite{weed2018explicit} on entropy regularization of finite-dimensional linear programs implies that $\alpha^\eps(p) = O\left( e^{-\Delta(p)/\eps} \right)$, where $\Delta(p) > 0$ is the suboptimality gap for the unregularized self-transport problem from $\mu_p$ to itself.
Given (\ref{eq:quantitative-bound}), this is
sufficient for deducing the convergence in (\ref{eq:conv-riem-metric}). 
However, in the case of self-transport, we can use the structure of the self-transport problem to further improve the exponent.

\begin{proposition}[Controlling the off-diagonal mass]
\label{prop: OfDiagMass}
Recall the definitions $\alpha^\eps(p) \vc = \sum_{k \neq l} \pi_{kl}^\eps(p)$ and $c(p) \vc = \min_{k \neq l} \|\iota_k(p) - \iota_l(p)\|^2$.
For each $p \in U$, it holds that 
\[
\alpha^\eps(p) \le N e^{- c(p)/\eps }.
\]
\end{proposition}

For example, if $a_k = \frac{1}{N}$ for all $k$, the suboptimality gap from \cite{weed2018explicit} is given by $\Delta(p) = \frac{2}{N} c(p)$, corresponding to permuting $k^*, l^*$ such that $(k^*, l^*) \in \argmin_{k \neq l} \|\iota_k(p) - \iota_l(p)\|^2$, and \cref{prop: OfDiagMass} is an improvement on the exponent for $N > 2$.

\begin{proof}
As the cost is zero on the diagonal and we consider self-transport, 
\[ \pi_{kk}^\eps = e^{2 \ff_k /\eps} \]
and thus
\[
\pi_{kl}^\eps = \sqrt{\pi_{kk}^\eps\pi_{ll}^\eps} \ e^{-\|\iota_k - \iota_l\|^2/\eps} \le \sqrt{a_k a_l}\ e^{-\|\iota_k - \iota_l\|^2/\eps} \leq  \sqrt{a_k a_l}\ e^{-c(p)/\eps},
\]
where the second inequality follows from the fact that an entry of the coupling cannot be larger than the corresponding row sum.
The claim follows by applying Cauchy-Schwarz to $\sum_{k \neq l} \sqrt{a_k a_l}$.
\end{proof}

Hence we will focus on proving the upper bound (\ref{eq:quantitative-bound}) in \cref{thm:conv-riem-metric}.
We break the proof down into two parts. First, in \cref{subsec:decomp}, we use the results in \cite{lavenant2025riemannian} to write
\begin{equation}
\label{eq:decomposition}
g^{\eps}_{ij} = \widetilde{g}^\eps_{ij} + r^\eps_{ij} + s^\eps_{ij}, 
\end{equation}
where $\widetilde{g}^\eps, r^\eps, s^\eps$ are given by
\begin{align}
\label{eq:def-g-r-s-eps}
\begin{split}
    \widetilde{g}^\eps_{ij} &\vc = \sum_{k, l} \pi_{kl}^\eps v_{k, i}^\top v_{l, j} \\
    r^\eps_{ij} &\vc = \frac{2}{\eps} \sum_{k, l} \pi_{kl}^\eps v_{k, i}^\top (\iota_l - m_k^\eps)(\iota_k - m_l^\eps)^\top v_{l, j}  \\
    s^\eps_{ij} &\vc = \frac{\eps}{2} \langle h_i^\eps, K^\eps(I-(K^\eps)^2)^\dagger h_j^\eps \rangle_a,
    \end{split}
\end{align}
where we recall from (\ref{eq:generalized-inverse}) that $A^\dagger$ denotes the generalized inverse of $A$ with respect to the \textit{$a$-weighted inner product $\langle \cdot, \cdot \rangle_a$}.
The terms $\widetilde{g}^\eps_{ij}$ and $r^\eps_{ij}$ are easily controlled; the $s^\eps_{ij}$ term is a bit trickier, as it involves a pseudo-inverse. We discuss how to control $s^\eps$ in \cref{subsec:s-eps}. 
Lastly, in \cref{subsec:cont}, we comment on which parts of our proof can be extended and which parts break down for measures with density.

\subsection{A decomposition of the Sinkhorn Riemannian metric}
\label{subsec:decomp}

In fact, the decomposition in (\ref{eq:decomposition}) is true more generally. Therefore, for this next lemma, we temporarily do not work under Assumption \ref{assump:1}.

\begin{lemma}
\label{lemma:decomp-general}

Let $\X \subset \R^d$ be the closure of a bounded open set. Suppose that $\{\mu_t\}_{t \in (-\delta, \delta)} \subset \mathcal{P}(\X)$ is continuously differentiable in $C^1(\X)^*$ with the weak-* topology, and that its derivative at $t = 0$ is given by $\dot{\mu} = - \dive(v \mu)$ for $\mu = \mu_0$ and a vector field $v \colon \X \to \R^d$ in $L^1(\X, \mu; \R^d).$ Then the following formulae hold
\begin{equation}
\label{eq:h-mu}
\frac{\eps}{2}H^\eps_\mu[\dot{\mu}](y) =  \int k^\eps_\mu(x, y) (y - T^\eps_{\mu}(x)) \cdot v(x) d\mu(x)
\end{equation}
\begin{equation}
    \label{eq:h-mu-norm}
    \frac{\eps}{2}\|H^\eps_\mu[\dot{\mu}]\|_{\mathcal{H}_\mu^\eps}^2 = \int v(x)^\top v(y)  + \frac{2}{\eps} \left(v(x)^\top \left( y - T^\eps_{\mu}(x)\right)\left(x - T^\eps_{\mu}(y) \right)^\top v(y) \right) d\pi_{\mu, \mu}^\eps(x, y),
\end{equation}
and we have the following decomposition for $g^\eps$:
\begin{equation}
\label{eq:decomp-general}
g^\eps_\mu (\dot{\mu}, \dot{\mu}) = \frac{\eps}{2}\left( \|H^\eps_\mu[\dot{\mu}]\|_{\mathcal{H}_\mu^\eps}^2 +  \langle H_\mu^\eps[\dot{\mu}], K_\mu^\eps (I- (K_\mu^\eps)^2)^{-1} H_\mu^\eps[\dot{\mu}] \rangle_{L^2(\mu)/\R} \right).
\end{equation}
\end{lemma}

\begin{proof}
(\ref{eq:h-mu}) and (\ref{eq:h-mu-norm}) are similar to computations already done in \cite{lavenant2025riemannian}. For example, for (\ref{eq:h-mu}), use
\[H^\eps_\mu[\dot{\mu}](y) = \langle \dot{\mu}, k_\mu^\eps(\cdot, y) \rangle = \int \nabla_xk_\mu^\eps (x, y) v(x) d\mu(x) \]
\[ \nabla_xk_\mu^\eps (x, y) = k_\mu^\eps(x, y) \left( \frac{\nabla f_{\mu, \mu}^\eps(x) - 2 (x-y)}{\eps} \right) \]
with the identity $2x - \nabla f^\eps_{\mu, \mu}(x) = 2 T^\eps_{\mu}(x)$; this is similar to the computation done in \cite[Proposition~5.14]{lavenant2025riemannian}. 
For
(\ref{eq:h-mu-norm}), 
use 
\[ \|H^\eps_\mu[\dot{\mu}]\|_{\mathcal{H}_\mu^\eps}^2 = \langle \dot{\mu}, H^\eps_\mu[\dot{\mu}] \rangle, \]
then the computations in the proof of \cite[Theorem 4.17]{lavenant2025riemannian}.
To prove (\ref{eq:decomp-general}), we expand
\begin{align*}
   g^\eps_\mu(\dot{\mu}, \dot{\mu}) &= \frac{\eps}{2} \langle \dot{\mu}, (\mathrm{id} - (K_\mu^\eps)^2)^{-1} H_\mu^\eps [\dot{\mu}] \rangle \\
    &= \frac{\eps}{2} \langle H_\mu^\eps[\dot{\mu}], (\mathrm{id} - (K_\mu^\eps)^2)^{-1} H_\mu^\eps [\dot{\mu}] \rangle_{\mathcal{H}_\mu^\eps/ \R} \\
    &= \frac{\eps}{2}\left( \|H_\mu^\eps[\dot{\mu}]\|_{\mathcal{H}_\mu^\eps}^2 + \langle H_\mu^\eps[\dot{\mu}], (\mathrm{id} - (K_\mu^\eps)^2)^{-1}(K_\mu^\eps)^{2} H_\mu^\eps[\dot{\mu}] \rangle_{\mathcal{H}_\mu^\eps/\R}\right),
\end{align*}
where we recall from \cite{lavenant2025riemannian} that $\mathcal{H}_\mu^\eps/\R$ is isometric to the subspace $\{\ind_\X\}^\perp$ orthogonal to the constant function $\ind_\X$ in $\mathcal{H}_\mu^\eps$, and $H_\mu^\eps[\dot{\mu}] \in \{\ind_\X\}^\perp$.
To conclude, recall that $K^\eps_\mu$ is a compact self-adjoint operator from $\mathcal{H}_\mu^\eps$ to itself and
\[ \langle K_\mu^\eps [\varphi], \psi \rangle_{\mathcal{H}_\mu^\eps} = \langle H_\mu^\eps [\varphi \mu], \psi \rangle_{\mathcal{H}_\mu^\eps} = \langle \varphi, \psi \rangle_{L^2(\mu)} \]
for each $\varphi, \psi \in \mathcal{H}_\mu^\eps$ \cite[Eq. 4.3]{lavenant2025riemannian}.
\end{proof}

Upon specializing \cref{lemma:decomp-general} to our assumptions in Assumption \ref{assump:1} and polarizing, we obtain the desired decomposition from (\ref{eq:decomposition}).
\begin{lemma}
\label{lemma:decomp}
For any $\eps > 0,$ we have the following formulae:
\begin{equation}
\label{eq:h-i-formula}
\frac{\eps}{2}h_i^\eps =  \bigg[ \sum_{l} K^\eps_{kl} \left(\iota_k - m_l^\eps \right) \cdot v_{l, i}  \bigg]_k 
\end{equation}
\begin{equation}
g^{\eps}_{ij} = \widetilde{g}^\eps_{ij} + r^\eps_{ij} + s^\eps_{ij}, 
\end{equation}
where
\begin{align*}
    \widetilde{g}^\eps_{ij} &\vc = \sum_{k, l} \pi_{kl}^\eps v_{k, i}^\top v_{l, j} \\
    r^\eps_{ij} &\vc = \frac{2}{\eps} \sum_{k, l} \pi_{kl}^\eps v_{k, i}^\top (\iota_l - m_k^\eps)(\iota_k - m_l^\eps)^\top v_{l, j}  \\
    s^\eps_{ij} &\vc = \frac{\eps}{2} \langle h_i^\eps, K^\eps(I-(K^\eps)^2)^\dagger h_j^\eps \rangle_a
\end{align*}
and $A^\dagger$ denotes the generalized inverse of $A$ with respect to $\langle \cdot, \cdot \rangle_a$, see (\ref{eq:generalized-inverse}).
\end{lemma}

From the definitions of $\alpha^\eps, B$ and $\widetilde{g}^\eps$, one immediately obtains
\begin{equation}
\label{eq:tilde-g-bound}
\left| \widetilde{g}_{ij}^\eps - g_{ij}^0 \right| \leq 2\alpha^\eps B^2. 
\end{equation}
As $m_k^\eps = \sum_l K^\eps_{kl} \iota_l$, we also have
\begin{align}
\label{eq:iota-m-bound}
\begin{split}
\left| \iota_k - m_k^\eps \right| &\leq D \sum_{\substack{l \in [N], \\ l \neq k}} K_{kl}^\eps \\
\left| \iota_l - m_k^\eps \right| &\leq D \quad \quad \text{for } k \neq l
\end{split}
\end{align}
from which we can bound
\begin{equation}
    \label{eq:r-bound}
    \frac{\eps}{2}\left| r_{ij}^\eps \right| \leq \alpha^\eps (BD)^2 + \sum_{k} \pi_{kk}^\eps \left( BD \sum_{\substack{l \in [N], \\ l \neq k}} K_{kl}^\eps \right)^2 \leq 2\alpha^\eps (BD)^2.
\end{equation}
As $\alpha^\eps$ converges exponentially to 0 as $\eps \to 0$ (\cref{prop: OfDiagMass}, or \cite{weed2018explicit}), this implies that for each $p \in U$, $\widetilde{g}^\eps(p) \to g^0(p)$ and $r^\eps(p) \to 0$ as $\eps \to 0$. Controlling $s^\eps$ is a bit trickier, as we need to (pseudo-)invert $I - (K^\eps)^2.$

\begin{remark}
Fix any $p \in U$.
Although the decomposition in \cref{lemma:decomp-general} does not require finitely supported measures, the convergence of our upper bound (\ref{eq:r-bound}) on $|r_{ij}^\eps(p)|$ to 0 as $\eps \to 0$ depends on $\alpha^\eps(p)$ being $o(\eps)$. Similarly, we will see that the convergence of our upper bound on $|s_{ij}^\eps(p)|$ to 0 as $\eps \to 0$ also depends on $\alpha^\eps(p)$ being $o(\eps)$ as $\eps\to 0$.
When our measures are not finitely supported, analogous statements may not be true. We will discuss this in more detail in \cref{subsec:cont}.
\end{remark}

\subsection{Bounding $s^\eps$}
\label{subsec:s-eps}
In this section, we prove the following bound on $s^\eps$.

\begin{proposition}
\label{prop:bound-s-eps}
Recall that $s^\eps$ is given by 
\[s^\eps_{ij} = \frac{\eps}{2} \langle h_i^\eps, K^\eps(I-(K^\eps)^2)^\dagger h_j^\eps \rangle_a,\]
where $A^\dagger$ denotes the generalized inverse of $A$ with respect to $\langle \cdot, \cdot \rangle_a$, see (\ref{eq:generalized-inverse}).
For all $i, j \in [m]$, we have
        \begin{equation}
        \label{eq:s-eps-rate}
        \left| s_{ij}^\eps \right| \leq \frac{2\alpha^\eps\left( BD \right)^2}{\eps}.
        \end{equation}
\end{proposition}
\cref{thm:conv-riem-metric} then follows from (\ref{eq:tilde-g-bound}), (\ref{eq:r-bound}) and \cref{prop:bound-s-eps}. Before proving \cref{prop:bound-s-eps}, we discuss why a crude bound with the spectral gap is in general not enough for proving the convergence of $s^\eps(p) \to 0$ as $\eps \to 0$ at $p \in U$.

\begin{remark}[An attempt at bounding $s^\eps$ via the spectral gap]
\label{rmk:spec-gap}
Fix any $p \in U$.
From \cite{lavenant2025riemannian}, $K^\eps(p)$ has eigenvalues satisfying $1 = \lambda_1^\eps > \lambda_2^\eps \geq \cdots  \geq \lambda_N^\eps \geq 0$, so that $1- \lambda_2^\eps$ corresponds to the spectral gap of the Markov chain given by $K^\eps(p).$
By (\ref{eq:h-i-formula}),
\[ \left\| \frac{\eps}{2} h_i^\eps(p) \right\|_a \lesssim \alpha^\eps(p) = O \left( e^{-c(p)/\eps} \right), \]
where the multiplicative constants can depend on $\iota(p)$ and $v_{k, i}(p)$ for $i\in [m], k \in [N]$.
As we treat $N$ and $a_1, \ldots, a_N$ as multiplicative constants as well, there exists $C > 0$, independent of $\eps$, such that
\[ \frac{1}{C}\alpha^\eps(p) \leq \left\| I - K^\eps(p) \right\|_{op} \leq C\alpha^\eps(p)\]
by the equivalence of norms on a finite-dimensional vector space.
Therefore a crude bound of $s^\eps_{ij}(p)$ is given by
\[ \left| s_{ij}^\eps (p) \right| \leq \frac{\lambda_2^\eps}{1- (\lambda_2^\eps)^2} \left\| \frac{\eps}{2} h_i^\eps(p) \right\|_a\left\|  h_j^\eps(p) \right\|_a  \lesssim \frac{\left\| I - K^\eps(p) \right\|_{op}^2}{\eps (1-\lambda_2^\eps)}. \]
If there exists $C' > 0$ such that $\left\| I - K^\eps(p) \right\|_{op} \leq C' (1-\lambda_2^\eps)$ for all $\eps > 0$ (e.g., if we only have $N = 2$ particles), then we obtain the same rate of convergence as (\ref{eq:s-eps-rate}), up to a multiplicative constant. More generally, if
\begin{equation}
\label{eq:spectral-gap-assump}
\lim_{\eps \to 0} \frac{e^{-2c(p)/\eps}}{\eps(1 - \lambda_2^\eps)} = 0
\end{equation}
then one can show that $s_{ij}^\eps(p) \to 0$ as $\eps \to 0$.
However, from \cite{lavenant2025riemannian} we only have the bound
\begin{equation}
\label{eq:spec-gap-lower-bound}
1-\lambda_2^\eps \geq e^{-4D(p)^2/\eps}, 
\end{equation}
i.e., a lower bound on the spectral gap which is controlled by the \emph{maximum} distance between two particles, whereas the upper bound on $\left\| I - K^\eps(p) \right\|_{op}$ is controlled by the \emph{minimum} distance between two particles. We will also see this behavior in the following example.
\end{remark}

\begin{example}[Example where the spectral gap assumption (\ref{eq:spectral-gap-assump}) does not hold]
\label{ex:violates-spec-gap}
Let $T > 0$, $L > \frac{\sqrt{2}}{2}$ and $\mu_t \vc = \frac{1}{4}\sum_{k = 1}^4 \delta_{\iota_k(t)}$ for $t \in (-T, T)$, where $\iota \colon (-T, T) \to (\R^3)^4$ is smooth and satisfies
\[\iota_1(0) = \left( -\frac{1}{2}, 0, 0 \right), \iota_2(0) = \left( \frac{1}{2}, 0, 0 \right), \iota_3(0) = \left(0, \frac{1}{2}, L \right), \iota_4(0) = \left(0, -\frac{1}{2}, L \right). \]
Then for $D^2 = \frac{1}{2} + L^2$,
\[ K^\eps(0) = \frac{1}{\beta^\eps}
\begin{bmatrix}

1 & \exp \left( - \frac{1}{\eps} \right) & \exp \left( - \frac{D^2}{\eps} \right) & \exp \left( - \frac{D^2}{\eps} \right) \\
\exp \left( - \frac{1}{\eps} \right) & 1 & \exp \left( - \frac{D^2}{\eps} \right)& \exp \left( - \frac{D^2}{\eps} \right)\\
\exp \left( - \frac{D^2}{\eps} \right)&\exp \left( - \frac{D^2}{\eps} \right) &1 & \exp \left( - \frac{1}{\eps} \right) \\
\exp \left( - \frac{D^2}{\eps} \right)&\exp \left( - \frac{D^2}{\eps} \right) & \exp \left( - \frac{1}{\eps} \right) & 1
\end{bmatrix},
\]
where $\beta^\eps \vc = 1 + \exp \left( - \frac{1}{\eps} \right) + 2 \exp \left( - \frac{D^2}{\eps} \right)$.
(As $L > \frac{\sqrt{2}}{2}$, $D > 1$ is $\max_{k, l} \|\iota_k(0) - \iota_l(0)\|$.) 
For all $\eps > 0$, $K^{(\eps)}(0)$ has eigenvectors
    \[ w_2 = \begin{bmatrix}
        1\\1\\-1\\-1
    \end{bmatrix}, \ w_3 = \begin{bmatrix}
        1 \\ -1 \\ -1 \\ 1
    \end{bmatrix} , \ w_4 = \begin{bmatrix}
        1 \\ -1 \\ 1 \\ -1
    \end{bmatrix} \]
corresponding to eigenvalues
\begin{align*}
    \lambda_2^\eps &= \frac{1}{\beta^\eps} \left( 1 + \exp \left( - \frac{1}{\eps} \right) - 2 \exp \left( - \frac{D^2}{\eps} \right) \right) = 1 - \frac{4\exp \left(-D^2/\eps \right)}{\beta^\eps} \\
    \lambda_3^\eps &= \lambda_4^\eps = \frac{1}{\beta^\eps} \left( 1- \exp\left( -\frac{1}{\eps} \right) \right).
\end{align*}
Since we can make $L$ (and hence $D$) arbitrarily large, (\ref{eq:spectral-gap-assump}) need not hold.
\end{example}

Hence a crude bound on $s_{ij}^\eps$ via the spectral gap will not be enough.
To get around this, the main observation is that $\frac{\eps}{2} h^\eps_i$ can be written as the graph divergence of an edge flow, whose norm is easy to control.

\begin{lemma}
\label{lemma:divergence}
    Fix any $\eps > 0$ and $p \in U$. Consider the graph $G$ with $N$ vertices and edge weights $W = \pi^\eps(p)$. 
    Then: 
    \begin{enumerate}
        \item The matrix $I - K^\eps(p)$ is precisely the graph Laplacian $L_G$ of $G$.
        \item For each $i \in [m]$,
    \[ \frac{\eps}{2} h_i^\eps(p) = \dive_G X^{i}(p),  \]
    where
    \[ X^{i}_{kl}(p) \vc = \left( \iota_k (p) - m_l^\eps(p) \right)\cdot v_{l, i}(p) - \left( \iota_l(p) - m_k^\eps(p) \right) \cdot v_{k, i}(p). \]

    \end{enumerate}

\end{lemma}
\begin{proof} 
    1. follows from the definitions of $K^\eps, \pi^\eps$ and the graph Laplacian $L_G$. For 2, by \cref{lemma:decomp},
    \[ \frac{\eps}{2}h_i^\eps =  \bigg[ \sum_{l} K^\eps_{kl} \left(\iota_k - m_l^\eps \right) \cdot v_{l, i}  \bigg]_k. \]
    On the other hand, $X^i(p)$ is anti-symmetric by construction, and the $k$th entry of $\dive_G X^i(p)$ is given by
    \begin{align*}
    (\dive_G X^i(p))_k &= \frac{1}{a_k} \sum_l \pi_{kl}^\eps(p) X^i_{kl}(p) \\
    &= \sum_l K^\eps_{kl}(p) \left( \left( \iota_k(p)  - m_l^\eps(p) \right)\cdot v_{l, i}(p) - \left( \iota_l(p) - m_k^\eps(p) \right) \cdot v_{k, i}(p) \right) \\
    &= \sum_l K^\eps_{kl}(p) \left( \iota_k(p)  - m_l^\eps(p) \right)\cdot v_{l, i}(p), 
    \end{align*}
    where the last line follows from the definition of $m^\eps$ (\ref{eq:m-eps}). 
\end{proof}

We will also need the following lemma.

\begin{lemma}
\label{lemma:grad-u-upper-bound}
Let $G = (V, E, W)$ be a graph with $|V|$ finite and symmetric, non-negative edge weights. Consider any two edge flows $X, Y \in L^2_\wedge (E)$, and define $u \vc = L_G^\dagger \dive_G Y,$ where $L_G^\dagger$ is the generalized inverse of $L_G$ as an operator from $L^2(V)$ to $L^2(V)$. Then
\begin{equation}
    \label{eq:upper-bound-gradu}
    \left\| \grad_G u \right\|_{L^2_\wedge(E)} \leq \left\| Y \right\|_{L^2_\wedge(E)} 
\end{equation}
    and we also have
\begin{equation}
\label{eq:neg-sob-norm-bound}
\left | \left \langle  \dive_G X, L_G^\dagger \dive_G Y  \right \rangle_{L^2(V)} \right| \leq \left\| X \right\|_{L^2_\wedge(E)} \left \| Y \right\|_{L^2_\wedge(E)}. 
\end{equation}
\end{lemma}

\begin{proof}
We claim that $\dive_G Y \in (\ker L_G)^\perp$, the orthogonal complement of $\ker L_G$ in $L^2(V)$. Indeed, if $v \in \ker L_G$, then
\[ 
\langle v, L_G v \rangle_{L^2(V)} = \left\langle \grad_G v, \grad_G v \right\rangle_{L^2_\wedge(E)} = 0, 
\]
so that $(\grad_G v)_{kl}$ is zero whenever the edge weight $W_{kl}$ is non-zero. Thus
\[
\langle v, \dive_G Y \rangle_{L^2(V)} = -\langle \grad_G v, Y \rangle_{L^2_\wedge(E)} = 0. 
\]
As $u  = L_G^\dagger \dive_G Y$, this implies that 
$L_G u = \dive_G Y$,
    so
    $Y + \grad_G u$ must belong to $\ker \dive_G$, which is orthogonal to $\Ima \grad_G$ as $\dive_G$ is the negative adjoint of $\grad_G$.
    Therefore (\ref{eq:upper-bound-gradu}) holds
    and we obtain
    \[ \left | \left \langle  \dive_G X, L_G^\dagger \dive_G Y  \right \rangle_{L^2(V)} \right| = \left| \langle X, \grad_G u \rangle_{L^2_\wedge(E)} \right| \leq \left\| X \right\|_{L^2_\wedge(E)} \left \| Y \right\|_{L^2_\wedge(E)} \]
    by Cauchy-Schwarz and (\ref{eq:upper-bound-gradu}).
\end{proof}

\begin{remark}
We remark on some analogues of the objects in \cref{lemma:grad-u-upper-bound}.
\begin{itemize}[label = {--}]
    \item For $f \in (\ker L_G)^\perp$, solving $L_G u = f$ is a graph version of solving a Poisson equation, and 
    \[ \left\| f \right\|_{\dot{H}^{-1}(G)} \vc = \left( \langle f, L_G^\dagger f \rangle_{L^2(V)} \right)^{1/2}\] can be viewed as a graph version of a homogeneous negative Sobolev norm. 
    Other works have studied similar objects, although sometimes with the unnormalized graph Laplacian, i.e.,  $L_G^{un} \vc = D_G-W$.
    (Recall that the graph Laplacian $L_G$ as we have defined it corresponds to the graph Laplacian with a random walk normalization (\ref{eq:gl-rw}).) For example, see \cite{calder2020poisson} for a discussion of the graph Poisson equation with $L_G^{un}$; \cite{bandeira2026mathematics} for the relationship between the pseudo-inverse of $L_G^{un}$, the commute time distance and effective resistance; and \cite{robertson2026resistance} for the connection with linearized optimal transport on a graph.
    \item One way to view (\ref{eq:upper-bound-gradu}) is via a graph version of the Helmholtz decomposition; see \cite{lim2020hodge}. (\ref{eq:upper-bound-gradu}) is also similar to why the tangent space of Wasserstein-2 space at $\mu \in \mathcal{P}_2(\R^d)$ is defined as
    \[T_\mu \mathcal{P}_2(\R^d) \vc = \overline{ \{ \nabla \varphi \colon \varphi \in C_c^\infty(\R^d) \}}^{L^2(\mu)}.\]
    Indeed, as noted in the introduction, $v \in T_\mu \mathcal{P}_2(\R^d)$ if and only if
\[
     \|v\|_{L^2(\mu)} \leq \left\| v+ w \right\|_{L^2(\mu)}
\]
for all vector fields $w \in L^2(\mu)$ such that $\dive(w\mu) = 0$ \cite[Lemma 8.4.2]{ambrosio2008gradient}.
\end{itemize}
    
\end{remark}


For convenience, let us rearrange
$s^\eps$ to
\begin{equation}
\label{eq:s-eps-rearr}
s^\eps_{ij}  = \frac{\eps}{2} \left \langle  (I + K^\eps)^{-1} K^\eps h_i^\eps, (I-K^\eps)^\dagger h_j^\eps \right \rangle_a. 
\end{equation}
At any fixed $p \in U$,
the spectrum of $K^\eps(p)$ lies in $[0, 1]$ \cite{lavenant2025riemannian}, so that $(I + K^\eps(p))^{-1}$ is well-defined.
Rearranging $s^\eps$ in this way is not necessary to show $s^\eps(p) \to 0$: we can use that
\begin{equation}
    \label{eq:comparison-w-neg-sob}
      \langle h, K^\eps(I-(K^\eps)^2)^\dagger h \rangle_a \leq \frac{1}{2} \langle h, (I-K^\eps)^\dagger h \rangle_a
\end{equation}
with $h = h_i^\eps$ or $h = h_i^\eps + h_j^\eps$, then apply \cref{lemma:divergence} and \cref{lemma:grad-u-upper-bound} and polarize.

Nonetheless, (\ref{eq:s-eps-rearr}) will be useful in \cref{sec:derivs}, where it is convenient to control the derivatives of $(I + K^\eps)^{-1} K^\eps  h_i^\eps$ and $(I-K^\eps)^\dagger h_j^\eps$ separately.
We now show that $\frac{\eps}{2} (I + K^\eps)^{-1} K^\eps h_i^\eps$ can also be expressed in terms of the graph divergence of an edge flow. 

\begin{lemma}
\label{lemma:mult-by-Keps-div}
Fix any $\eps > 0$ and $p \in U$. Consider the graph $G$ with $N$ vertices and edge weights $W = \pi^\eps(p)$. Then for any edge flow $X \in L^2_\wedge(E)$,
\[ K^\eps(p) \dive_G X = \dive_G Y, \quad (I+ K^\eps(p))^{-1} \dive_G X = \dive_G Z, \]
where edge flows $Y, Z \in L^2_\wedge(E)$ are given by
\begin{align*}
    Y &= (I + \grad_G \dive_G) X \\
    Z &= (2I + \grad_G \dive_G )^{-1}X.
\end{align*}
In particular, $\grad_G \dive_G$ is a self-adjoint operator from $L^2_\wedge(E)$ to itself with eigenvalues lying in $[-1, 0]$, so that $2I + \grad_G \dive_G$ is indeed invertible, and the operator norm of \[(2I + \grad_G \dive_G)^{-1} \colon L^2_\wedge(E) \to L^2_\wedge(E)\] is at most 1.
\end{lemma}
\begin{proof}
    To see that $K^\eps(p) \dive_G X = \dive_G Y$, recall that $K^\eps(p) = I - L_G = I + \dive_G \grad_G$. 
    The self-adjointness of $\grad_G \dive_G$ is clear from $\grad_G = - \dive_G^*$. Additionally, $\grad_G = - \dive_G^*$ implies that the non-zero eigenvalues of $\grad_G \dive_G$ agree with the non-zero eigenvalues of \[\dive_G \grad_G = - L_G = K^\eps(p) - I.\] As shown in \cite{lavenant2025riemannian}, the eigenvalues of $K^\eps(p)$ lie in $[0, 1]$, so the eigenvalues of $\grad_G \dive_G$ lie in $[-1, 0]$. Therefore $2I + \grad_G \dive_G$ is invertible, the eigenvalues of $(2I + \grad_G \dive_G)^{-1}$ lie in $\left[\frac{1}{2}, 1\right]$ and
    \[ \dive_G X = \dive_G (2I + \grad_G \dive_G)Z = (2I - L_G) \dive_G Z = (I + K^\eps(p))\dive_G Z. \qedhere
    \]
\end{proof}

It remains to prove \cref{prop:bound-s-eps}.

\begin{proof}[Proof of \cref{prop:bound-s-eps}]
Fix any $p \in U$ and let $G$ be the graph with edge weights $W = \pi^\eps(p).$
    By Lemmas \ref{lemma:divergence} and \ref{lemma:mult-by-Keps-div},
    \begin{align*}
        \frac{\eps}{2} (I + K^\eps)^{-1} K^\eps h_i^\eps(p) &= \dive_G (2I + \grad_G \dive_G)^{-1} (I + \grad_G \dive_G) X^i(p) \\
        \frac{\eps}{2} h_j^\eps(p) &= \dive_G X^j(p),
    \end{align*}
    where $X^{i}_{kl}(p) \vc = \left( \iota_k (p) - m_l^\eps(p) \right)\cdot v_{l, i}(p) - \left( \iota_l(p) - m_k^\eps(p) \right) \cdot v_{k, i}(p).$
    Therefore, by \cref{lemma:grad-u-upper-bound},
    \begin{align*}
    \frac{\eps}{2} \left|s^\eps_{ij}\right|  &= \left(\frac{\eps}{2}\right)^2 \left| \left \langle  (I + K^\eps)^{-1} K^\eps h_i^\eps, (I-K^\eps)^\dagger h_j^\eps \right \rangle_a \right| \\
    &\leq \left\| (2I + \grad_G \dive_G)^{-1} (I + \grad_G \dive_G) X^i(p) \right\|_{L^2_\wedge(E)} \left\| X^j(p) \right\|_{L^2_\wedge(E)}.
    \end{align*}
    As noted in \cref{lemma:mult-by-Keps-div}, the eigenvalues of $\grad_G \dive_G$ lie in $[-1, 0]$, so that
    \[ \left\| (2I + \grad_G \dive_G)^{-1} (I + \grad_G \dive_G) X^i(p) \right\|_{L^2_\wedge(E)}\leq \frac{1}{2}\left\| X^i(p) \right\|_{L^2_\wedge(E)}. \]
    We can also bound
    \[ \left\| X^i(p) \right\|^2_{L^2_\wedge(E)} = \sum_{k < l} \pi_{kl}^\eps \left( X_{kl}^i \right)^2 \leq 2\alpha^\eps (BD)^2 \]
    for each $i$, so that altogether
    \[  \frac{\eps}{2} \left|s^\eps_{ij}\right| \leq \frac{1}{2} \left\| X^i(p) \right\|_{L^2_\wedge(E)}\left\| X^j(p) \right\|_{L^2_\wedge(E)} \leq \alpha^\eps (BD)^2. \qedhere
    \]
\end{proof}

\subsection{Comparison with the absolutely continuous setting}
\label{subsec:cont}

Here we highlight where our proof breaks down for measures with density. 
Consider the assumptions for \cref{lemma:decomp-general}: $\{\mu_t\}_{t \in (-\delta, \delta)} \subset \mathcal{P}(\X)$ continuously differentiable in $C^1(\X)^*$ with the weak-* topology, with its derivative at $t = 0$ given by $\dot{\mu} = - \dive(v \mu)$ for $\mu = \mu_0$ and a vector field $v \in L^1(\X, \mu; \R^d).$ 
Analogous to the finitely supported case, define
\begin{align*}
    \widetilde{g}^\eps_\mu(\dot{\mu}, \dot{\mu}) &\vc = \int v(x)^\top v(y) d\pi_{\mu, \mu}^\eps(x, y) \\
    r^\eps_\mu(\dot{\mu}, \dot{\mu}) &\vc = \frac{2}{\eps} \int \left(v(x)^\top \left( y - T^\eps_{\mu}(x)\right)\left(x - T^\eps_{\mu}(y) \right)^\top v(y) \right) d\pi_{\mu, \mu}^\eps(x, y) \\
    s^\eps_\mu(\dot{\mu}, \dot{\mu}) &\vc = \frac{\eps}{2}\langle H^\eps_\mu[\dot{\mu}], K^\eps_\mu (I- (K^\eps_\mu)^2)^{-1} H^\eps_\mu[\dot{\mu}] \rangle_{L^2(\mu)/\R},
\end{align*}
so that \cref{lemma:decomp-general} reads
\[ g_\mu^\eps(\dot{\mu}, \dot{\mu}) = \widetilde{g}^\eps_\mu(\dot{\mu}, \dot{\mu}) +  r^\eps_\mu(\dot{\mu}, \dot{\mu}) + s^\eps_\mu(\dot{\mu}, \dot{\mu}). \]
Recall that we did \emph{not} assume that $\mu$ is finitely supported for \cref{lemma:decomp-general}.

However, in general, we do not expect $r_\mu^\eps(\dot{\mu}, \dot{\mu}) \to 0$ or $s^\eps_\mu(\dot{\mu}, \dot{\mu}) \to 0$ as $\eps \to 0$ when $\mu$ has a density. The most obvious issue is when $\dot{\mu} = -\dive(v\mu) = 0$, but $v$ is continuous and non-zero. Then
\[ \frac{\eps}{2}\|H^\eps_\mu[\dot{\mu}]\|_{\mathcal{H}_\mu^\eps}^2 = \widetilde{g}^\eps_\mu(\dot{\mu}, \dot{\mu}) + r_\mu^\eps(\dot{\mu}, \dot{\mu}) = 0 \]
but whenever $\mu$ has a density (with respect to the Lebesgue measure) and finite entropy,
\[ \lim_{\eps \to 0} \widetilde{g}^\eps_\mu(\dot{\mu}, \dot{\mu}) = \int \|v\|^2 d\mu \]
by the narrow convergence of $\pi_{\mu, \mu}^\eps$ as $\eps \to 0$
\cite{carlier2017convergence}.

Certainly, to avoid this issue, we could take $v$ from the closure of $\{\nabla \varphi \colon \varphi \in C_c^\infty(\R^d)\}$ in $L^2(\mu)$, the tangent space of Wasserstein-2 space. Still, we claim that in general we should not expect $r_\mu^\eps(\dot{\mu}, \dot{\mu}) \to 0$ or $s^\eps_\mu(\dot{\mu}, \dot{\mu}) \to 0$ as $\eps \to 0$.
To further highlight the differences between the density and finite support cases, we recall the informal argument given in \cite{lavenant2025riemannian} for the convergence of $g_\mu^\eps(\dot{\mu}, \dot{\mu}) \to g_\mu^0(\dot{\mu}, \dot{\mu})$ as $\eps \to 0$ when $\mu$ has a smooth enough density. Using a Laplace expansion and the Schr\"odinger equation for the self-transport potential, it is argued in \cite{lavenant2025riemannian} that we should have
\begin{equation}
\label{eq:k-mu-density-case}
\varphi - K_\mu^\eps[\varphi] = \frac{\eps}{4} \left( \Delta \varphi - \nabla \log \mu \cdot \nabla \varphi \right) + o(\eps),
\end{equation}
where we identify $\mu$ with its density. (Beware the differing Laplacian conventions: \cite{lavenant2025riemannian} takes $\Delta$ to be the trace of the Hessian, whereas we use the negative of the trace of the Hessian. The latter agrees with our convention for the graph Laplacians.)
Using (\ref{eq:k-mu-density-case}), \cite{lavenant2025riemannian} then argues that we should have
\[ g^\eps_\mu(\dot{\mu}, \dot{\mu}) = \frac{\eps}{2} \left\langle \dot{\mu}, \left( \mathrm{id} - (K_\mu^\eps)^2\right)^{-1} K_\mu^\eps \left[ \frac{\dot{\mu}}{\mu} \right] \right\rangle = \left\langle \frac{\dot{\mu}}{\mu}, L^{-1} \left[ \frac{\dot{\mu}}{\mu} \right]\right \rangle_{L^2(\mu)} + o(1)  \]
for $L\varphi = \Delta \varphi - \nabla \log \mu \cdot \nabla \varphi$. 
(\ref{eq:k-mu-density-case}) is in contrast to when $\mu$ has finite support, in which case $\left| \varphi - K_\mu^\eps[\varphi] \right|$ can be bounded in terms of the off-diagonal mass, which converges at a rate of $O\left( e^{-c/\eps} \right)$ for some $c>0$ (\cref{prop: OfDiagMass}, or \cite{weed2018explicit}). In fact, this was crucial for the finite support case: we showed that $\left|r^\eps_{ij}\right|, \left|s^\eps_{ij}\right| \lesssim \alpha^\eps/\eps$, which would not imply $r_{ij}^\eps(p), s_{ij}^\eps(p) \to 0$ for each $p \in U$ if we only had $\alpha^\eps(p) = O(\eps).$

As a concrete example, consider when $v$ is constant. By a slight abuse of notation, we also treat $v$ as a vector in $\R^d$. In this case, from \cite[Proposition 5.14]{lavenant2025riemannian}, we have
\[ g^\eps_\mu(\dot{\mu}, \dot{\mu}) = \|v\|^2 \]
for all $\eps > 0$,
which agrees with what one would expect from the Wasserstein-2 Riemannian structure.
Moreover, we also have 
\[\widetilde{g}^\eps_\mu(\dot{\mu}, \dot{\mu}) = \|v\|^2 \]
so that $r^\eps_\mu(\dot{\mu}, \dot{\mu}) = - s^\eps_\mu(\dot{\mu}, \dot{\mu})$ in the constant velocity setting. If $v \neq 0$, under mild assumptions on the density of $\mu$, $\left|r_\mu^\eps(\dot{\mu}, \dot{\mu})\right| = \Omega(1)$ as $\eps \to 0$.

\begin{proposition}
\label{prop:r-eps-cont-setting}
Suppose $\Omega \subset \R^d$ is a bounded open set with $C^1$ boundary, and $\mu\in \mathcal{P}(\overline{\Omega})$ has density $\rho$ (with respect to the Lebesgue measure) satisfying $\rho(x) > 0$ for $x \in \Omega$, $\rho(x) = 0$ for $x \in \partial \Omega$ and $\widetilde{\rho} \vc= \sqrt{\rho} \in C^2(\overline{\Omega})$.
Then for $\dot{\mu} = -\dive(v \mu)$ with any fixed $v \in \R^d\setminus\{0\}$,
\begin{equation}
\label{eq:r-eps-covariance-err}
r^\eps_\mu(\dot{\mu}, \dot{\mu})  = -\frac{2}{\eps} v^\top \left(\int \left( y - T_\mu^\eps(x)\right)\left( y - T_\mu^\eps(x)\right)^\top d\pi_{\mu, \mu}^\eps(x, y) \right) v + O(\eps^{1/2}) 
\end{equation}
\begin{equation}
\label{eq:covar-asymp}
\frac{1}{\eps} v^\top \left(\int \left( y - T_\mu^\eps(x)\right)\left( y - T_\mu^\eps(x)\right)^\top d\pi_{\mu, \mu}^\eps(x, y) \right) v = \Theta(1) 
\end{equation}
as $\eps \to 0$.
\end{proposition}

The rate in (\ref{eq:covar-asymp}) agrees with some intuition from the existing literature, e.g., \cite{lavenant2025riemannian, mordant2024entropic}, where one expects $\exp(f_{\mu, \mu}^\eps/\eps) \sim (\pi \eps)^{-d/4} \mu^{-1/2}$, so that informally $\kk_{\mu}^\eps(x, y) d\mu(x)$ should behave like a Gaussian centered at $x$ with covariance $\frac{\eps}{2} I$.

\begin{proof}
Let us prove the $O(1)$ upper bound for (\ref{eq:covar-asymp}) first.
We use \cite[Lemma 1]{pmlr-v162-pooladian22a}, which states that
\[
\int \left\| \nabla f^\eps_{\mu, \mu} \right\|^2d\mu \leq \frac{\eps^2}{4} I(\mu), 
\]
where $I(\mu) \vc = 4 \int \|\nabla \sqrt{\rho}\|^2 $ is the Fisher information of $\mu$, which is well-defined and finite under our assumptions (see Appendix \ref{app:cramer-rao}). Recalling that $2x - \nabla f_{\mu, \mu}^\eps(x) = 2 T_\mu^\eps(x)$, this implies that
\begin{equation}
\label{eq:diff-with-T-eps}
\int \left\| x- T_\mu^\eps(x) \right\|^2 d\mu(x) \leq \frac{\eps^2}{16} I(\mu) 
\end{equation}
and so by Cauchy-Schwarz
\begin{align*}
   \left| \int \int y \cdot (y-x) k^\eps_\mu(x, y) d\mu(x) d\mu(y) \right |&=\left| \int y \cdot (y - T_\mu^\eps(y)) d\mu(y) \right| \\
    &\leq \left( \int \|y\|^2 d\mu(y) \right)^{1/2} \left( \int \|y- T_\mu^\eps(y)\|^2 d\mu(y) \right)^{1/2} \\
    &= O(\eps).
\end{align*}
Therefore, using a symmetry argument, 
$\int \|y-x\|^2 d\pi_{\mu, \mu}^\eps(x, y) = O(\eps)$ and also
\begin{equation}
\label{eq:y-t-eps-x-bound}
\int \|y- T_\mu^\eps(x)\|^2 d\pi_{\mu, \mu}^\eps(x, y) \leq 2 \left(\int \|y-x\|^2 + \left\| x- T_\mu^\eps(x) \right\|^2 d\pi_{\mu, \mu}^\eps(x, y)  \right) = O(\eps), \end{equation}
which implies the $O(1)$ upper bound in (\ref{eq:covar-asymp}). 
As
\[ r^\eps_\mu(\dot{\mu}, \dot{\mu})  = \frac{2}{\eps} v^\top \left(\int  \left( y - T^\eps_{\mu}(x)\right)\left(x - T^\eps_{\mu}(y) \right)^\top d\pi_{\mu, \mu}^\eps(x, y)\right)v \]
under the constant velocity assumption, (\ref{eq:r-eps-covariance-err}) follows from the previous bounds and another application of Cauchy-Schwarz:
\begin{align*}
    \bigg| r_\mu^\eps(\dot{\mu}, \dot{\mu} ) &+ \frac{2}{\eps} v^\top \left(\int \left( y - T_\mu^\eps(x)\right)\left( y - T_\mu^\eps(x)\right)^\top d\pi_{\mu, \mu}^\eps(x, y) \right) v \bigg| \\
    &= \frac{2}{\eps} \left| v^\top \left(\int \left( y - T_\mu^\eps(x)\right)\left( y - T_\mu^\eps(y) + x- T_\mu^\eps(x)\right)^\top d\pi_{\mu, \mu}^\eps(x, y) \right) v  \right| \\
    &\leq \frac{4}{\eps} \left( \int \left| v^\top  \left( y - T_\mu^\eps(x)\right)\right|^2 d\pi_{\mu, \mu}^\eps(x, y) \right)^{1/2} \left( \int \left| \left( x - T_\mu^\eps(x)\right)^\top v  \right|^2 d\pi_{\mu, \mu}^\eps(x, y)\right)^{1/2}
\end{align*}
which is $O(\eps^{1/2})$ by (\ref{eq:diff-with-T-eps}) and (\ref{eq:y-t-eps-x-bound}).

It remains to prove the lower bound of $\Omega(1)$ for (\ref{eq:covar-asymp}). This is inspired by \cite{chewi2023entropic}, where they prove Caffarelli's contraction theorem by using the Hessian of the potentials and covariance inequalities and then taking $\eps\to 0$. Let us define $W(x) \vc= -\log \rho(x) = - 2 \log \widetilde{\rho}(x) $ for $x \in \Omega$. One of the observations in \cite{chewi2023entropic} (and earlier in \cite[Appendix~B.2]{daniels2021score}), specialized to the self-transport setting, is that for $\varphi^\eps(x) \vc = \|x\|^2 - f^\eps_{\mu, \mu}(x)$,
\begin{equation}
\label{eq:hess-potential}
\nabla^2 \varphi^\eps(x) = \frac{4}{\eps} \int (y - T_\mu^\eps(x)) (y- T_\mu^\eps(x))^\top \kk^\eps_\mu(x, y) d\mu(y) = \frac{4}{\eps} \mathrm{Cov}_{Y \sim \pi_x^\eps}(Y) 
\end{equation}
\begin{equation}
-\nabla^2 \log \pi_x^\eps (y) = - \left( \frac{\nabla^2 f_{\mu, \mu}^\eps(y) - 2I }{\eps} \right) + \nabla^2 W(y) = \frac{1}{\eps} \nabla^2 \varphi^\eps(y) + \nabla^2 W(y) 
\end{equation}
for $x, y \in \Omega$,
where the measure $\pi_x^\eps$ is given by $d\pi_x^\eps(y) = \kk_\mu^\eps(x, y) d\mu(y).$ (The factor of $4$ arises from our use of the cost $x, y \mapsto \|x-y\|^2$, compared to the convention in \cite{chewi2023entropic} of $x, y \mapsto \frac{1}{2}\|x-y\|^2$.)
As in \cite{chewi2023entropic}, using the Cram\'er-Rao inequality, we bound
\begin{align*}
\frac{1}{\eps} \mathrm{Cov}_{Y \sim \pi_x^\eps}(Y) &\succeq \frac{1}{\eps}\left( \E_{Y \sim \pi_x^\eps} \left[ \frac{1}{\eps} \nabla^2 \varphi^\eps(Y) + \nabla^2 W(Y) \right] \right)^{-1} \\
&= \left( \E_{Y \sim \pi_x^\eps} \left[\nabla^2 \varphi^\eps(Y)  + \eps \nabla^2 W(Y)  \right] \right)^{-1}.
\end{align*}
See Appendix \ref{app:cramer-rao} for a more detailed justification.
Since \cite{chewi2023entropic} would like bounds on $\nabla^2\varphi^\eps(x)$ for all $x$, they then use a uniform upper bound on $\nabla^2 W$ and the Brascamp--Lieb inequality (which requires a strictly log-concave measure) to bound the last line. However, as $W(x)$ goes to infinity as $x$ approaches $\partial \Omega$, we do not have $\beta > 0$ such that $\nabla^2 W(x) \preceq \beta I$ for all $x \in \Omega$.
Instead, as we are only interested in
\[  \int \left( y - T_\mu^\eps(x)\right)\left( y - T_\mu^\eps(x)\right)^\top d\pi_{\mu, \mu}^\eps(x, y) = \E_{X \sim \mu}\left[\mathrm{Cov}_{Y \sim \pi_X^\eps}(Y) \right], \]
we can use the convexity of the matrix inverse on positive definite matrices to bound
\begin{align*}
\frac{1}{\eps} \E_{X \sim \mu} \left[w^\top \mathrm{Cov}_{Y \sim \pi_X^\eps}(Y) w\right] &\geq \E_{X \sim \mu} \left[ w^\top\left(\E_{Y \sim \pi_X^\eps} \left[\nabla^2 \varphi^\eps(Y)  + \eps  \nabla^2 W(Y)  \right]\right)^{-1}w\right] \\ 
&\geq   w^\top \left(\E_{X \sim \mu} \left[\E_{Y \sim \pi_X^\eps} \left[\nabla^2 \varphi^\eps(Y)  + \eps  \nabla^2 W(Y)  \right]\right]\right)^{-1} w \\
&= w^\top \left(\E_{ Y \sim \mu} \left[ \nabla^2 \varphi^\eps(Y)  + \eps  \nabla^2 W(Y)  \right]\right)^{-1}w
\end{align*}
for any $w \in \R^d$, where we use Jensen's inequality in the second line and Fubini's theorem in the last line. One can check that each expectation (integral) above is well-defined, and Fubini's theorem can be applied in the last line: by the integrability of $\nabla^2 W$ with respect to $\mu$ (\cref{lemma:cramer-rao}), $x, y \mapsto \left( \nabla^2 \varphi^\eps(y) + \eps \nabla^2 W(y) \right) k_\mu^\eps(x, y)$ is integrable with respect to $\mu \otimes \mu$, and
\[ x \mapsto \E_{Y \sim \pi_x^\eps} \left[\nabla^2 \varphi^\eps(Y)  + \eps  \nabla^2 W(Y)  \right] = \int \left( \nabla^2 \varphi^\eps(y) + \eps \nabla^2 W(y) \right) k_\mu^\eps(x, y) d\mu(y) \]
is continuous on $\Omega$ and also integrable with respect to $\mu$.
By (\ref{eq:hess-potential}),
\[ \E_{X \sim \mu} \left[ \nabla^2 \varphi^\eps(X) \right] = \frac{4}{\eps}\E_{X \sim \mu} \left[ \mathrm{Cov}_{Y \sim \pi_X^\eps}(Y) \right]. \]
Moreover, under our assumptions on $\mu$,
$\beta \vc= \|\E_{\mu}\left[  \nabla^2 W \right]\|_{op}$
is finite.
Therefore 
\[ \frac{1}{\eps} \E_{X \sim \mu} \left[ \mathrm{Cov}_{Y \sim \pi_X^\eps}(Y) \right] \succeq \left(\frac{4}{\eps} \E_{X \sim \mu} \left[ \mathrm{Cov}_{Y \sim \pi_X^\eps}(Y) \right] + \eps  \beta  \right)^{-1}, \]
and any eigenvalue $\lambda$ of $\frac{1}{\eps} \E_{X \sim \mu} \left[\mathrm{Cov}_{Y \sim \pi_X^\eps}(Y)\right]$ satisfies
\[ \lambda (4\lambda + \eps \beta) \geq 1, \]
which gives us
\[ \lambda \geq \sqrt{\frac{1}{4} + \left( \frac{\eps \beta}{8} \right)^2} - \frac{\eps \beta}{8} \]
since $\lambda$ must be non-negative.
As a result,
\[ \frac{1}{\eps} v^\top \E_{X \sim \mu} \left[\mathrm{Cov}_{Y \sim \pi_X^\eps}(Y)\right] v = \Omega(1) \]
as $\eps \to 0$, proving (\ref{eq:covar-asymp}).
\end{proof}

\section{Convergence of higher order derivatives}\
\label{sec:derivs}
A similar proof strategy allows us to establish the convergence of the derivatives of $g^{\eps}$ as $\eps \to 0$.
\begin{theorem}
\label{thm:derivs-of-g}
Fix any $\eps_0 > 0$ and $n \in \Z_{> 0}$. Then there exists $C_{n+1} \colon U \to \R_{> 0}$, which can be written as a continuous function of $\iota$ and its partial derivatives of order $\leq n+1$, such that
    \[ \left| \frac{\partial^n g_{ij}^\eps }{\partial p_{i_1} \cdots \partial p_{i_n}} - \frac{\partial^n g_{ij}^0 }{\partial p_{i_1} \cdots \partial p_{i_n}}\right| \leq \frac{C_{n+1} \alpha^\eps}{\eps^{n+1}} \]
    for all $\eps \in (0, \eps_0)$ and $i_1, \ldots, i_n \in [m]$. Therefore,
    \begin{equation}
    \label{eq:derivs-conv}
    \lim_{\eps \to 0} \frac{\partial^n g_{ij}^\eps }{\partial p_{i_1} \cdots \partial p_{i_n}}(p) = \frac{\partial^n g_{ij}^0 }{\partial p_{i_1} \cdots \partial p_{i_n}}(p) 
    \end{equation}
    for all $p \in U$.
\end{theorem}
\textbf{For the rest of this section, $\eps_0 > 0$ will be fixed, and $C_n \colon U \to \R_{> 0}$ will always denote a function which can be written as a continuous function of $\iota$ and its partial derivatives of order $\leq n$, though the precise function may change from line to line.} The functions $C_n$ may depend on the choice of $\eps_0$. In particular, the convergence in (\ref{eq:derivs-conv}) can be made uniform over compact sets of $U$.

\begin{remark}
Since the Christoffel symbols and Riemann curvature tensor for a Riemannian metric $g$ can be expressed in terms of the derivatives of $g$ and its inverse, \cref{thm:derivs-of-g} immediately implies the pointwise convergence of the Christoffel symbols and Riemann curvature tensor for $g^{\eps}$ as $\eps \to 0$, which can be made uniform over compact subsets of $U$.
\end{remark}

To prove \cref{thm:derivs-of-g}, recall that the main difficulty in \cref{sec:conv} was bounding $s^\eps$.
Similarly, bounding the derivatives of $s^\eps$ with respect to $p$ will be the crux of proving \cref{thm:derivs-of-g}.
To take derivatives of (\ref{eq:s-eps-rearr}) with respect to $p$, recall that if $A \colon (-\delta, \delta) \to \R^{N \times N}$ is smooth and $A(t)$ is invertible for all $t \in (-\delta, \delta)$, then
\[ 
\frac{dA^{-1}}{dt} = - A^{-1} \frac{dA}{dt} A^{-1}. \]
With the product rule, 
we can then obtain formulae for the partial derivatives of $A \colon U \to \R^{N \times N}$ if $A$ is smooth and $A(p)$ is invertible for all $p \in U$: for any $i_1, \ldots, i_n \in [m]$,
\begin{equation}
    \label{eq:partial-derivs-of-inverse}
    \frac{\partial^n A^{-1}}{\partial p_{i_1} \cdots \partial p_{i_n}} = \sum_{D = 1}^n \sum_{\substack{\text{ordered} \\ \text{partitions } J \\ \text{of } \{i_1, \ldots, i_n\}, \\ |J| = D}} (-1)^D A^{-1} \prod_{d = 1}^D \left(\frac{\partial^{|J_d|} A}{\partial p_{J_d}} A^{-1} \right),
\end{equation}
By ordered partition, we mean that $(\{i_1, i_3\}, \{i_2\})$ is a different partition than $( \{i_2\}, \{i_1, i_3\})$, though $(\{i_1, i_3\}, \{i_2\})$ and $(\{i_3, i_1\}, \{i_2\})$ are the same partition. 

In addition to applying (\ref{eq:partial-derivs-of-inverse}) to $A = I + K^\eps$,
we can also apply it to $A = I - K^\eps$ restricted to the orthogonal complement $\ind^\perp$ of $\mathrm{span}(\{\ind\})$ with respect to $\langle \cdot, \cdot \rangle_a$. Indeed, for all $p \in U$, $I - K^\eps(p)$ is invertible when restricted to $\ind^\perp$. To be more concrete, one could fix an orthonormal basis $\{e_1, \ldots, e_{N- 1}\}$ of $\ind^\perp$ and take $A \colon U \to \R^{(N -1) \times (N-1)}$ to be $I - K^\eps$ with respect to this basis.

The main observation is that the order of the factors in
\begin{equation}
\label{def:P_J}
P_J \vc = \prod_{d = 1}^{|J|} \left(\frac{\partial^{|J_d|} A}{\partial p_{J_d}} A^{-1}\right)
\end{equation}
alternates between $A^{-1}$ and partial derivatives of $A$. If $A = I - K^\eps$ (restricted to $\ind^\perp$) or $A = I + K^\eps$, and $G$ is the graph with edge weights $\pi^\eps(p)$ for some $p \in U$, then for any edge flow $X$, this alternating pattern will allow us to write $P_J(p) \dive_G X$ as the divergence of another edge flow.

To prove this, we will need the following proposition on the derivatives of $\pi^\eps$ and $K^\eps$. We recall that for $\eps > 0$, $\ff$ and hence $\pi^\eps$ are smooth by an implicit function theorem argument, similar to \cite{lavenant2025riemannian, luise2018differential, carlier2024displacement, xu2026convergence}.

\begin{proposition}[Derivatives of $\pi^\eps$ and $K^\eps$ with respect to $p$]
\label{prop:pi-derivs}
    The partial derivatives of $\pi^\eps$ with respect to $p$ can be written in the form
    \begin{equation}
    \label{eq:pi-derivs}
        \frac{\partial^n \pi_{kl}^\eps}{\partial p_{i_1} \cdots \partial p_{i_n}} = \frac{1}{\eps^n} \pi_{kl}^\eps \psi^\eps_{kl, i_1, \ldots, i_n}
    \end{equation}
    for some smooth $\psi^\eps_{kl, i_1, \ldots, i_n} \colon U \to \R$ with $\psi^\eps_{kl, i_1, \ldots, i_n} = \psi^\eps_{lk, i_1, \ldots, i_n}$. Moreover, there exists  a function $C_n \colon U \to \R_{> 0}$, which can be written as a continuous function of $\iota$ and its partial derivatives of order $\leq n$, such that
    \begin{equation}
    \label{eq:psi-bound}
        \left|\psi^\eps_{kl, i_1, \ldots, i_n}\right| \leq C_n
    \end{equation}
    for all $\eps \in (0, \eps_0)$.
    Since $K^\eps_{kl} = \frac{1}{a_k} \pi^\eps_{kl}$, we also have
    \begin{equation}
    \label{eq:k-derivs}
    \frac{\partial^n K_{kl}^\eps}{\partial p_{i_1} \cdots \partial p_{i_n}} = \frac{1}{\eps^n} K_{kl}^\eps \psi^\eps_{kl, i_1, \ldots, i_n}.
    \end{equation}

\end{proposition}

Before proving \cref{prop:pi-derivs}, we discuss several corollaries and their implications for proving \cref{thm:derivs-of-g}.

\begin{corollary}
\label{cor:deriv-of-graph-div}
    Suppose that $X \colon U \to \R^{N \times N}$ is smooth with $X_{kl} = - X_{lk}$, and fix any $\eps > 0$. Consider
    $\eta_k \vc = \sum_l K^\eps_{kl}X_{kl}$ and any $i_1, \ldots, i_n \in [m]$.  Then
    \[ \frac{\partial^n \eta_k}{\partial p_{i_1} \cdots \partial p_{i_n}} = \sum_l \sum_{S \subseteq [n]} \frac{\partial^{|S|} K_{kl}^\eps}{\partial p_{i_S}} \frac{\partial^{|[n] \setminus S|}X_{kl}}{\partial p_{i_{[n] \setminus S}}} =  \sum_l K_{kl}^\eps \sum_{S \subseteq [n]} \eps^{ - |S|}\psi^\eps_{kl, i_S} \frac{\partial^{|[n] \setminus S|}X_{kl}}{\partial p_{i_{[n] \setminus S}}}. \]
    In other words, if $\eta_k(p) = (\dive_G X(p))_k$ for each $p \in U$ and graph $G$ with edge weights $W = \pi^\eps(p)$, then for each $p \in U$
    \[ \frac{\partial^n \eta_k}{\partial p_{i_1} \cdots \partial p_{i_n}}(p) =  (\dive_G Z(p))_k, \]
    where again $G$ is the graph with edge weights $W = \pi^\eps(p)$ and $Z$ is given by
    \[Z_{kl}(p) = \sum_{S \subseteq [n]} \eps^{ - |S|}\psi^\eps_{kl, i_S}(p) \frac{\partial^{|[n] \setminus S|}X_{kl}}{\partial p_{i_{[n] \setminus S}}}(p).\]
\end{corollary}
\cref{cor:deriv-of-graph-div} follows immediately from (\ref{eq:k-derivs}) and allows us handle the derivatives of $h_i^\eps$.
\cref{prop:pi-derivs} can also be used to write the derivatives of $(I + K^\eps)^{-1} h_i^\eps$ and $(I - K^\eps)^\dagger h_i^\eps$ as the divergence of an edge flow.

\begin{corollary}
\label{cor:K-eps-deriv-divergence}
Fix any $\eps > 0$, $p \in U$ and $i_1, \ldots, i_n \in [m]$. Consider the graph $G$ with $N$ vertices and edge weights $W = \pi^\eps(p)$. Then for any $w  \in \R^N$, 
\[ \bigg[\frac{\partial^n K^\eps}{\partial p_{i_1} \cdots \partial p_{i_n}}(p)\bigg] w = \dive_G Z, \]
where the edge flow $Z$ is given by \[Z_{kl} \vc = \frac{1}{\eps^n} \psi^\eps_{kl, i_1, \ldots, i_n}(p) (\grad_G w)_{kl}.\]
\end{corollary}

\begin{proof}[Proof of \cref{cor:K-eps-deriv-divergence}]
    From \cref{prop:pi-derivs}, 
    \[ \frac{\partial^n K^\eps}{\partial p_{i_1} \cdots \partial p_{i_n}}w = \frac{1}{\eps^{n}}  \left[  \sum_l K_{kl}^\eps \psi^\eps_{kl, i_1, \ldots, i_n}w_l \right]_k. \]  
    Observe that for any $k$,
    \[  \frac{1}{\eps^n} \sum_l K_{kl}^\eps \psi_{kl, i_1, \ldots, i_n}^\eps = \sum_l \frac{\partial^n K_{kl}^\eps}{\partial p_{i_1} \cdots \partial p_{i_n}} = 0 \]
    since $\sum_l K_{kl}^\eps = 1$. Therefore
    \[ \frac{1}{\eps^n}\sum_l K_{kl}^\eps(p) \psi^\eps_{kl, i_1, \ldots, i_n}(p)w_l = \frac{1}{\eps^n}\sum_l K_{kl}^\eps(p) \psi^\eps_{kl, i_1, \ldots, i_n}(p)(w_l - w_k) = (\dive_G Z)_k. \qedhere
    \] 
\end{proof}
\begin{corollary}
\label{cor:p-j-divergence}
Fix any $\eps \in (0, \eps_0)$, $p \in U$ and $i_1, \ldots, i_n \in [m]$. Consider the graph $G$ with $N$ vertices and edge weights $W = \pi^\eps(p)$. For each ordered partition $J$ of $\{i_1, \ldots, i_n \}$, define $P_J$ as in (\ref{def:P_J}), where either
\begin{enumerate}[label = (\roman*)]
    \item $A = I - K^\eps$ (restricted to $\ind^\perp$), or
    \item $A = I + K^\eps$.
\end{enumerate}
Then there exists $C_n \colon U \to \R_{> 0}$, which can be written as a continuous function of $\iota$ and its partial derivatives of order $\leq n$, such that for any $X \in L^2_\wedge(E)$ and ordered partition $J$ of $\{i_1, \ldots, i_n \}$, there exists $Z \in L^2_\wedge(E)$ satisfying
\[ P_J(p) \dive_G X = \dive_G Z  \]    
whose norm is bounded by
\[\left\|Z \right\|_{L^2_\wedge(E)} \leq \frac{C_n(p)}{\eps^n} \left\|X \right\|_{L^2_\wedge(E)}. \]
\end{corollary}
Recall that we use $C_n \colon U \to \R_{> 0}$ to denote a function that can be expressed as a continuous function of $\iota$ and its partial derivatives of order $\leq n$, but that the precise function may change from line to line.
\begin{proof}
First consider when $A = I - K^\eps$, restricted to $\ind^\perp$. For an edge flow $Y$, 
\[ \langle \dive_G Y, \ind \rangle_a = -\langle Y, \grad_G \ind \rangle_{L^2_\wedge(E)} = 0, \]
so $\dive_G Y \in \ind^\perp.$
Let $w \vc= (I- K^\eps(p))^{-1} \dive_G Y,$ where $(I- K^\eps(p))^{-1}$ is understood as an operator from $\ind^\perp$ to itself.
By \cref{cor:K-eps-deriv-divergence},
\begin{equation}
\label{eq:one-step-div}
\left[\frac{\partial^{|J_d|}K^\eps}{\partial p_{J_d}} (p) \right]w = \dive_G Z^d , \quad Z^d_{kl} \vc = \frac{1}{\eps^{|J_d|}} \psi^\eps_{kl, J_d}(p) (\grad_G w)_{kl}.
\end{equation}
   Now use (\ref{eq:psi-bound}) to obtain
    \[ \left\| Z^d \right\|_{L^2_\wedge (E)} \leq \frac{C_{|J_d|}(p)}{\eps^{|J_d|}} \left\| \grad_G w \right\|_{L^2_\wedge (E)} \leq \frac{C_{|J_d|}(p)}{\eps^{|J_d|}}  \left\| Y \right\|_{L^2_\wedge (E)}, \]
    where the last inequality follows from \cref{lemma:grad-u-upper-bound}. The corollary statement for $A = I - K^\eps$ then follows from induction.

    Suppose instead that $A = I + K^\eps$ and $w = (I+ K^\eps(p))^{-1} \dive_G Y$. (\ref{eq:one-step-div}) still holds, and by \cref{lemma:mult-by-Keps-div}, 
    $w = \dive_G \widehat{Y}$, where $\widehat{Y} \vc= (2I + \grad_G \dive_G)^{-1}Y$. Therefore
    \begin{align*}
    \left\| Z^d \right\|_{L^2_\wedge (E)} &\leq \frac{C_{|J_d|}(p)}{\eps^{|J_d|}} \left\| \grad_G w \right\|_{L^2_\wedge (E)} \\
    &= \frac{C_{|J_d|}(p)}{\eps^{|J_d|}} \left \| \grad_G \dive_G (2I + \grad_G \dive_G)^{-1}Y \right\|_{L^2_\wedge(E)} \\
    &\leq \frac{C_{|J_d|}(p)}{\eps^{|J_d|}} \left \| Y \right\|_{L^2_\wedge(E)},
    \end{align*}
    where we recall from \cref{lemma:mult-by-Keps-div} that the eigenvalues of $\grad_G \dive_G$ lie in $[-1, 0]$. The corollary again follows from induction.
\end{proof}

Given \cref{prop:pi-derivs} and its corollaries, we have everything we need to prove \cref{thm:derivs-of-g}.
\begin{proof}[Proof of \cref{thm:derivs-of-g}]
    Fix any $p \in U$. Recall the decomposition $g^{\eps}_{ij} = \widetilde{g}^\eps_{ij} + r^\eps_{ij} + s^\eps_{ij}$ from \cref{lemma:decomp}, where
\begin{align*}
    \widetilde{g}^\eps_{ij} &\vc = \sum_{k, l} \pi_{kl}^\eps v_{k, i}^\top v_{l, j} \\
    r^\eps_{ij} &\vc = \frac{2}{\eps} \sum_{k, l} \pi_{kl}^\eps v_{k, i}^\top (\iota_l - m_k^\eps)(\iota_k - m_l^\eps)^\top v_{l, j}  \\
    s^\eps_{ij} &\vc = \frac{\eps}{2} \langle h_i^\eps, K^\eps(I-(K^\eps)^2)^\dagger h_j^\eps \rangle_a
\end{align*}
and $A^\dagger$ denotes the generalized inverse of $A$ with respect to $\langle \cdot, \cdot \rangle_a$, as in (\ref{eq:generalized-inverse}).
    As before, $C_n \colon U \to \R_{> 0}$ will always denote a function that can be expressed as a continuous function of $\iota$ and its partial derivatives of order $\leq n$, but the precise function may change from line to line.
    
    ~\\
    \textbf{Derivatives of $\widetilde{g}^\eps$.} 
    Since $g^0_{ij} = \sum_{k} a_k v_{k, i}^\top v_{k, j}$,
    \[ \frac{\partial^n g^0_{ij}}{\partial p_{i_1} \dots \partial p_{i_n}} = \sum_k a_k \frac{\partial^n[v_{k, i}^\top v_{k, j}]}{\partial p_{i_1} \dots \partial p_{i_n}}.  \]
    Upon repeated applications of the product rule to $\widetilde{g}^\eps_{ij}$ and using \cref{prop:pi-derivs}, we have that there exists a $C_{n+1}$ such that
    \[ \left| \frac{\partial^n \widetilde{g}^\eps_{ij}}{\partial p_{i_1} \dots \partial p_{i_n}} - \frac{\partial^n g^0_{ij}}{\partial p_{i_1} \dots \partial p_{i_n}} \right| \leq  \frac{C_{n+1} \alpha^\eps}{\eps^n} \]
    for all $\eps \in (0, \eps_0).$ Indeed, it is easy to see how to bound
    \[ \frac{\partial^n \pi_{kl}^\eps}{\partial p_{i_1} \cdots \partial p_{i_n}} = \frac{1}{\eps^n} \pi_{kl}^\eps \psi^\eps_{kl, i_1, \ldots, i_n} \]
    whenever $k \neq l$
    in terms of $\alpha^\eps$ and the $C_n$ from \cref{prop:pi-derivs}. As $\sum_l \pi_{kl}^\eps = a_k$, one can also bound the on-diagonal terms by using
    \[ \frac{\partial^n \pi_{kk}^\eps}{\partial p_{i_1} \cdots \partial p_{i_n}} = -\sum_{\substack{l \in [N],\\ l \neq k}} \frac{\partial^n \pi_{kl}^\eps}{\partial p_{i_1} \cdots \partial p_{i_n}} \]
    whenever $n \geq 1$.
    
    
    ~\\
    \textbf{Derivatives of $r^\eps$.} 
    Since $m_k^\eps = \frac{1}{a_k} \sum_l \pi_{kl}^\eps \iota_l$, again by \cref{prop:pi-derivs}, there exists a $C_n$ satisfying
    \begin{equation}
    \label{eq:derivs-of-m}
    \left\|\frac{\partial^n m_k^\eps}{\partial p_{i_1} \cdots \partial p_{i_n}} - \frac{\partial^n \iota_k}{\partial p_{i_1} \cdots \partial p_{i_n}}\right\| \leq \frac{C_n \alpha^\eps}{\eps^n} 
    \end{equation}
    for all $\eps \in (0, \eps_0)$. 
    This controls the derivatives of $\pi_{kk}^\eps (\iota_k - m_k^\eps)(\iota_k - m_k^\eps)^\top$: for some $C_n$,
    \[ \left\| \frac{\partial^n [\pi_{kk}^\eps (\iota_k - m_k^\eps)(\iota_k - m_k^\eps)^\top]}{\partial p_{i_1} \cdots \partial p_{i_n}} \right\|_{op} \leq \frac{C_n (\alpha^\eps)^2}{\eps^n} \]
    for all $\eps \in (0, \eps_0)$, where $\| \cdot \|_{op}$ denotes the operator norm.
    By (\ref{eq:derivs-of-m}), we also have
    \begin{equation}
    \label{eq:off-diag-difference}
    \left\|\frac{\partial^n m_k^\eps}{\partial p_{i_1} \cdots \partial p_{i_n}} - \frac{\partial^n \iota_l}{\partial p_{i_1} \cdots \partial p_{i_n}}\right\| \leq \left\|\frac{\partial^n \iota_k}{\partial p_{i_1} \cdots \partial p_{i_n}} - \frac{\partial^n \iota_l}{\partial p_{i_1} \cdots \partial p_{i_n}}\right\| + \frac{C_n \alpha^\eps}{\eps^n}.
    \end{equation}
    The derivatives of $\pi_{kl}^\eps (\iota_l - m_k^\eps)(\iota_k - m_l^\eps)^\top$ for $k \neq l$ can be controlled using (\ref{eq:off-diag-difference}) and upper bounds on $\pi_{kl}^\eps$ and its derivatives, so that altogether we have the existence of a $C_n$ satisfying
    \[ \left\| \frac{\partial^n [\pi_{kl}^\eps (\iota_l - m_k^\eps)(\iota_k - m_l^\eps)^\top]}{\partial p_{i_1} \cdots \partial p_{i_n}} \right\|_{op} \leq \frac{C_n \alpha^\eps}{\eps^n} \]
    for all $k, l \in [N]$ and $\eps \in (0, \eps_0)$. One then obtains, for some $C_{n+1}$,
    \[ \left| \frac{\partial^n r_{ij}^\eps}{\partial p_{i_1} \cdots \partial p_{i_n}}  \right| \leq  \frac{C_{n+1}\alpha^\eps}{\eps^{n+1}} \]
    for all $\eps \in (0, \eps_0)$.

    ~\\
    \textbf{Derivatives of $s^\eps$.} Let us define $\vv^\eps_i \colon U \to \R^N$ and $\ww^\eps_i \colon U \to \R^N$ by
    \begin{align*}
    \vv^\eps_i &\vc = \frac{\eps}{2} (I + K^\eps)^{-1} K^\eps h_i^\eps \\ \ww^\eps_i &\vc = \frac{\eps}{2} (I-K^\eps)^\dagger h_i^\eps
    \end{align*}
    so that
    \[s^\eps_{ij}  = \frac{\eps}{2} \left \langle  (I + K^\eps)^{-1} K^\eps h_i^\eps, (I-K^\eps)^\dagger h_j^\eps \right \rangle_a = \frac{2}{\eps} \langle \vv_i^\eps, \ww_j^\eps \rangle_a.\]
    Recall from \cref{lemma:divergence} that for each $p \in U$, we have
    \begin{align*}
    \vv^\eps_i(p) &=  (I + K^\eps(p))^{-1} K^\eps(p) \dive_G X^i(p) =  \dive_G X^i(p) - (I + K^\eps(p))^{-1}  \dive_G X^i(p) \\ 
    \ww^\eps_i(p) &=  (I-K^\eps(p))^\dagger \dive_G X^i(p) 
    \end{align*}
    where $G$ is the graph with edge weights $W = \pi^\eps(p)$ and
    $X^{i}_{kl} = \left( \iota_k - m_l^\eps \right)\cdot v_{l, i} - \left( \iota_l - m_k^\eps \right) \cdot v_{k, i}.$
    Using (\ref{eq:off-diag-difference}), one obtains some $C_{n+1}$ satisfying
    \[ \left| \frac{\partial^n X^{i}_{kl}}{\partial p_{i_1} \cdots \partial p_{i_n}} \right| \leq C_{n+1} \]
    for all $k \neq l$ and $\eps \in (0, \eps_0)$. 
    For each fixed $p \in U$,
    taking $\eta_k(p) = (\dive_G X^i(p))_k$ in \cref{cor:deriv-of-graph-div}, we have that
    \[ \frac{\partial^n \eta_k}{\partial p_{i_1} \cdots \partial p_{i_n}}(p) = (\dive_G Z^i(p))_k \]
    where the edge flow $Z^i(p)$ is given by
    \[Z_{kl}^i(p) = \sum_{S \subseteq [n]} \eps^{ - |S|}\psi^\eps_{kl, i_S}(p) \frac{\partial^{|[n] \setminus S|}X_{kl}^i}{\partial p_{i_{[n] \setminus S}}}(p),\]
    so that there exists $C_{n+1}$ satisfying
    \[ \left\| Z^i(p) \right\|_{L^2_\wedge(E)} = \left( \sum_{k < l} \pi^\eps_{kl}(p)  Z^i_{kl}(p)^2 \right)^{1/2} \leq \frac{ C_{n+1}(p) \sqrt{\alpha^\eps}}{\eps^n} \]
    for all $\eps \in (0, \eps_0)$. As we have
    \[ \vv^\eps_i(p) =  \dive_G X^i(p) - (I + K^\eps(p))^{-1}  \dive_G X^i(p), \]
    we can use the product rule and (\ref{eq:partial-derivs-of-inverse}) with \cref{cor:p-j-divergence} and \cref{lemma:mult-by-Keps-div} to obtain an edge flow $Z^{\vv_i^\eps}$ such that
     \[ \dive_G Z^{\vv_i^\eps} = \frac{\partial^n \vv_i^\eps}{\partial p_{i_1} \cdots \partial p_{i_n}} (p), \quad \left \| Z^{\vv_i^\eps} \right\|_{L^2_\wedge(E)} \leq \frac{C_{n+1}(p)\sqrt{\alpha^\eps}}{\eps^n} \]
     for some $C_{n+1}$.
    Similarly, for $\ww^\eps_i(p) =  (I-K^\eps(p))^\dagger \dive_G X^i(p)$,
    we have some $C_{n+1}$ and an edge flow $Z^{\ww_i^\eps}$ such that 
    \[ (I - K^\eps(p))^\dagger\dive_G Z^{\ww_i^\eps} = \frac{\partial^n \ww_i^\eps}{\partial p_{i_1} \cdots \partial p_{i_n}} (p), \quad \left \| Z^{\ww_i^\eps} \right\|_{L^2_\wedge(E)} \leq \frac{C_{n+1}(p)\sqrt{\alpha^\eps}}{\eps^n}\]
    for all $\eps \in (0, \eps_0)$.
    Applying the product rule to $ \frac{\eps}{2} s_{ij}^\eps =  \langle \vv_i^\eps, \ww_j^\eps \rangle$ and using \cref{lemma:grad-u-upper-bound} then gives us, for some $C_{n+1}$,
     \[ \frac{\eps}{2} \left| \frac{\partial^n s_{ij}^\eps}{\partial p_{i_1} \cdots \partial p_{i_n}} \right| \leq  \frac{C_{n+1}\alpha^\eps}{\eps^{n}} \]
    for all $\eps \in (0, \eps_0)$.
\end{proof}

We finally turn to the proof of \cref{prop:pi-derivs}.
This is done with the following lemma and an induction argument.

\begin{lemma}[First order derivatives of $\ff$ with respect to $p$]
\label{lemma:ff-derivs}
    The first partial derivatives of $\ff$ are given by
    \begin{equation}
    \label{eq:first-partial-ff}
        \frac{\partial \ff}{\partial p_i} = 
        2\left( I + K^\eps \right)^{-1} \left( \bigg[ \sum_l K^\eps_{kl} \langle \iota_k - \iota_l , v_{k, i} - v_{l, i} \rangle \bigg]_k \right).
    \end{equation}
    Therefore
        \begin{equation}
        \label{eq:potential-deriv-upper-bound}
        \left\| \frac{\partial \ff}{\partial p_i} \right\|_a \leq 2 \left\| \bigg[ \sum_l K^\eps_{kl} \langle \iota_k - \iota_l , v_{k, i} - v_{l, i} \rangle \bigg]_k \right\|_a \leq 4 \left( \left( \sum_k \frac{1}{a_k} \right)^{1/2} \alpha^\eps B D \right).
    \end{equation}
\end{lemma}

\begin{proof}
    The Schr\"odinger system implies that
    \[ \ff_k = -\eps \log \left( \sum_l a_l \exp \left( \frac{\ff_l - \|\iota_k - \iota_l\|^2}{\eps} \right) \right). \]
    Differentiating with respect to $p_i$ then gives us
    \begin{align*}
        \frac{\partial \ff_k}{\partial p_i} &= - \frac{\sum_l a_l \exp \left( \frac{\ff_l - \|\iota_k - \iota_l\|^2}{\eps} \right) \left( \frac{\partial \ff_l}{\partial p_i} - 2\langle \iota_k - \iota_l, v_{k, i} - v_{l, i} \rangle \right) }{\sum_l a_l \exp \left( \frac{\ff_l - \|\iota_k - \iota_l\|^2}{\eps} \right)} \\
        &= - \sum_l a_l \exp \left( \frac{\ff_k + \ff_l - \|\iota_k - \iota_l\|^2}{\eps} \right) \left( \frac{\partial \ff_l}{\partial p_i} - 2\langle \iota_k - \iota_l, v_{k, i} - v_{l, i} \rangle \right) \\
        &= -\sum_l K^\eps_{kl} \left( \frac{\partial \ff_l}{\partial p_i} - 2\langle \iota_k - \iota_l, v_{k, i} - v_{l, i} \rangle \right).
    \end{align*}
    (\ref{eq:first-partial-ff}) and the first inequality of (\ref{eq:potential-deriv-upper-bound}) follow from rearranging and recalling that the spectrum of $K^\eps(p)$ lies in $[0, 1]$ for all $p \in U$ \cite{lavenant2025riemannian}.
    The second inequality of (\ref{eq:potential-deriv-upper-bound}) follows from observing that for each fixed $k$,
    \[
   \left| \sum_l K^\eps_{kl} \langle \iota_k - \iota_l , v_{k, i} - v_{l, i} \rangle \right|  = \frac{1}{a_k} \left| \sum_{l \neq k} \pi^\eps_{kl} \langle \iota_k - \iota_l , v_{k, i} - v_{l, i} \rangle \right| \leq \frac{2 \alpha^\eps B D}{a_k}.
   \]
\end{proof}

\begin{proof}[Proof of \cref{prop:pi-derivs}]
        By the chain rule, the first partial derivatives of the self-transport plan $\pi^\eps$ are given by
    \begin{equation}
        \frac{\partial \pi_{kl}^\eps}{\partial p_i} = \frac{1}{\eps} \pi_{kl}^\eps \left( \frac{\partial \ff_k}{\partial p_i} + \frac{\partial \ff_l}{\partial p_i} - 2\langle \iota_k - \iota_l, v_{k, i} - v_{l, i} \rangle \right).
    \end{equation}
    Then by \cref{lemma:ff-derivs}, the proposition statement holds for $n = 1$. Though (\ref{eq:potential-deriv-upper-bound}) is stated in terms of $\|\cdot\|_a$, one easily obtains entry-wise bounds by
$a_k \left|\frac{\partial \ff_k}{\partial p_i}  \right|^2 \leq  \left\| \frac{\partial \ff}{\partial p_i} \right\|_a^2.$

    Now consider $n \geq 2$ and suppose the proposition statement holds for all partial derivatives of order $< n$. For convenience, define
    \[ F_{kl}^\eps \vc = \ff_k + \ff_l - \|\iota_k - \iota_l \|^2. \]
    By Fa\`a di Bruno's formula (i.e., repeatedly applying the chain and product rules),
    we have
    \[\frac{\partial^n \pi^\eps_{kl}}{\partial p_{i_1} \cdots \partial p_{i_n}} = \frac{1}{\eps^n} \pi^\eps_{kl}\sum_J \eps^{n-|J|}  \prod_{d = 1}^{|J|} \frac{\partial^{|J_d|}F_{kl}^\eps}{\partial p_{J_d}} \]
    where the sum is over (unordered) partitions $J$ of $\{i_1, \ldots, i_n\}$. 
    For any subset $I'$ of $\{i_1, \ldots, i_n\}$, one can differentiate (\ref{eq:first-partial-ff}) and write 
    \[ p \mapsto \frac{\partial^{|I'|} \ff}{\partial p_{I'}}(p)\] 
    in terms of partial derivatives of $\pi^\eps$ only of order $\leq n-1$ and derivatives of $\iota$ of order $\leq n$. Then by applying the inductive hypothesis, one can bound $\left\|\frac{\partial^{|I'|}\ff}{\partial p_{I'}}\right\|_a$ uniformly across $\eps \in (0, \eps_0)$ by some $C_n$ depending continuously on $\iota$ and its derivatives of order $\leq n$.
    Therefore the absolute value of
    \[ \psi^\eps_{kl, i_1, \ldots, i_n}(p) \vc = \sum_J \eps^{n-|J|}  \prod_{d = 1}^{|J|} \frac{\partial^{|J_d|}F_{kl}^\eps}{\partial p_{J_d}}(p) \]
    can also be bounded uniformly across $\eps \in (0,\eps_0)$ with some such $C_n$.
\end{proof}




\section*{AI disclosure}
AI tools were used to assist with this work.
In particular, we had observed that $\widetilde{g}^\eps(p) \to g(p) $ and $r^\eps(p) \to 0$ as $\eps \to 0$, and that $s^\eps(p) \to 0$ as $\eps \to 0$ under the  assumption (\ref{eq:spectral-gap-assump}) on the spectral gap of the Markov chain with transitions given by $K^\eps(p)$. Based on an initial expansion of the potentials that we came up with, Claude Opus 4.8 gave us a proof under an assumption on the spectral gap of a slightly different matrix. It suggested \cref{ex:violates-spec-gap} as an example which did \textit{not} satisfy its assumptions. 

However, upon computing the eigendecomposition of $K^\eps(0)$ in \cref{ex:violates-spec-gap}, we saw that it could be used to prove the desired convergence of $s^\eps(0) \to 0$ for \cref{ex:violates-spec-gap}. Therefore, we suspected that a spectral gap assumption was in fact not necessary, but knew that using the spectral gap to bound the generalized inverse of $I - K^\eps$ (as in \cref{rmk:spec-gap}) would not be enough.

Given this information, ChatGPT 5.6 Sol recognized that $\frac{\eps}{2} h_i^\eps$ could be written as the graph divergence of an edge flow (\cref{lemma:divergence}), and that this could be used to prove the convergence of $s^\eps(p) \to 0$ as $\eps \to 0$ via a Cauchy-Schwarz argument and (\ref{eq:comparison-w-neg-sob}). It also proved the quantitative bound in \cref{thm:conv-riem-metric} when $\iota$ is a curve (i.e., $U = (-\delta, \delta)$ for some $\delta > 0$); as noted in \cref{sec:conv}, we can use a polarization argument to apply this to more general open sets $U \subset \R^m$. However, we observed that a similar approach could be used to prove the convergence of the derivatives of $g^\eps$ as $\eps \to 0$ (\cref{thm:derivs-of-g}) and opted to slightly modify the argument in \cref{sec:conv} because of this (see the discussion before \cref{lemma:mult-by-Keps-div}). LLM-generated contributions have since been restructured and rewritten, and we take responsibility for the correctness of the proofs.

\section*{Acknowledgments}
L.X. was supported by
the National Science Foundation Graduate Research Fellowship Program under Grant No. DGE-2444107
and the Simons Foundation Math+X Investigator Award to Amit Singer. Any opinions,
findings, and conclusions or recommendations expressed in this material are those of the
authors and do not necessarily reflect the views of the National Science Foundation. 

\bibliographystyle{abbrv}
\bibliography{paper-refs}

\appendix

\section{On the regularity conditions for \cref{prop:r-eps-cont-setting}}
\label{app:cramer-rao}

Here we provide a more detailed justification that our regularity assumptions for \cref{prop:r-eps-cont-setting} are sufficient.
Specifically, the proof of \cref{prop:r-eps-cont-setting} requires 
\begin{enumerate}[label = (\roman*)]
    \item the Fisher information of $\mu$ to be well-defined and finite, and
    \item for a Cram\'er-Rao type lower bound to hold for each $\pi_x^\eps$ given by $d\pi_x^\eps(y) = k_\mu^\eps(x, y) d\mu(y)$.
\end{enumerate} 
If $\mu$ has a density $\rho$ with respect to the Lebesgue measure and $\sqrt{\rho} \in H^1(\R^d)$, then its Fisher information
\begin{equation}
\label{eq:fi}
I(\mu) \vc = 4\int_{\R^d} \left\| \nabla \sqrt{\rho} \right\|^2 dx 
\end{equation}
is finite. In particular, $\sqrt{\rho} \in H^1(\R^d)$ requires $\sqrt{\rho}$ to be weakly differentiable, i.e., for an integration by parts identity to hold.
The proof of Cram\'er-Rao also relies heavily on integration by parts. On a bounded domain, integration by parts is Green's identity, which typically involves a boundary integral unless appropriate boundary conditions are imposed. We argue that both (i) and (ii) hold under the assumptions on $\mu$ in \cref{prop:r-eps-cont-setting}, which we repeat here for convenience.

\begin{assumption}
\label{assump:mu-density}
    Suppose $\Omega \subset \R^d$ is a bounded open set with $C^1$ boundary, and $\mu\in \mathcal{P}(\overline{\Omega})$ has density $\rho$ (with respect to the Lebesgue measure) satisfying:
\begin{itemize}[label = {--}]
    \item $\rho(x) > 0$ for $x \in \Omega$,
    \item $\rho(x) = 0$ for $x \in \partial \Omega$,
    \item $\widetilde{\rho} \vc= \sqrt{\rho} \in C^2(\overline{\Omega})$.
\end{itemize}
We also define $W(x) \vc= -\log \rho(x) = - 2 \log \widetilde{\rho}(x) $ for $x \in \Omega$, i.e., $\rho(x) = e^{-W(x)}$ on $\Omega$. 
\end{assumption}

Due to the boundary condition $\widetilde{\rho}(x) = 0$ for $x \in \partial \Omega$, we have, for any $\varphi \in C^1(\overline{\Omega})$,
\[ \int_\Omega \varphi(x) \nabla \widetilde{\rho}(x) dx = - \int_\Omega \widetilde{\rho}(x) \nabla \varphi(x) dx \]
by Green's identity (e.g., \cite[Appendix C.2]{evans}). Therefore $\widetilde{\rho} = \sqrt{\rho}$ (extended naturally to $\R^d$ by $\widetilde{\rho}(x) = 0$ for $x \notin \overline{\Omega}$) is weakly differentiable and, as $\Omega$ is bounded, $\widetilde{\rho} \in H^1(\R^d)$.
In fact, this only required $\widetilde{\rho} \in C^1(\overline{\Omega})$. 

We now show that a Cram\'er-Rao lower bound for $\mu$ holds under Assumption \ref{assump:mu-density}.

\begin{lemma}
    \label{lemma:cramer-rao}
    Under Assumption \ref{assump:mu-density}, $\nabla^2W$ is integrable with respect to $\mu$, $\E_\mu\left[\nabla^2W \right]$ is positive definite, and
\begin{equation}
\label{eq:cramer-rao}
\mathrm{Cov}_{X \sim \mu}(X) \succeq \left( \E_\mu \left[\nabla^2W \right] \right)^{-1}. 
\end{equation}
\end{lemma}

\begin{remark}
Recall that we would like to apply Cram\'er-Rao to $d\pi_x^\eps = k_\mu^\eps(x, y) d\mu(y)$. If $\mu$ satisfies Assumption \ref{assump:mu-density}, then $\pi_x^\eps$ has density 
\[y \mapsto k_\mu^\eps(x, y) \rho(y) = \exp \left( \frac{f_{\mu, \mu}^\eps(x) + f_{\mu, \mu}^\eps(y) - \|x-y\|^2}{\eps} \right) \rho(y).\]
Since $f_{\mu, \mu}^\eps$ is smooth (recall the representation \eqref{eq:schrodinger-1-1}), each $\pi_x^\eps$ satisfies Assumption \ref{assump:mu-density} whenever $\mu$ does, so we can apply \cref{lemma:cramer-rao} to $\pi_x^\eps$.
\end{remark}

The proof is similar to \cite{chewi2023entropic}. However, as we are working with a bounded domain $\Omega$ and $W$ goes to infinity as one approaches the boundary $\partial \Omega$, we provide additional details to justify the Cram\'er-Rao inequality under Assumption \ref{assump:mu-density}.

\begin{proof}[Proof of \cref{lemma:cramer-rao}]
First we check that $\nabla^2 W$ is integrable with respect to $\mu$. For this, it is convenient to write the derivatives of $W = - 2 \log \widetilde{\rho}$ in terms of $\widetilde{\rho}$:
\[ \frac{\partial W}{\partial x_k} = -\frac{2}{\widetilde{\rho}} \frac{\partial \widetilde{\rho}}{\partial x_k} \]
\[
\frac{\partial^2 W}{\partial x_k \partial x_l} = -2 \left( -\frac{1}{\widetilde{\rho}^2} \frac{\partial \widetilde{\rho}}{\partial x_l} \frac{\partial \widetilde{\rho}}{\partial x_k} + \frac{1}{\widetilde{\rho}} \frac{\partial^2 \widetilde{\rho}}{\partial x_k \partial x_l} \right) 
\]
on $\Omega$.
As $\widetilde{\rho}^2 = \rho$ and $\widetilde{\rho} \in C^2(\overline{\Omega})$, we have that each $\frac{\partial^2 W}{\partial x_k \partial x_l}$ is integrable with respect to $\mu$ and
\begin{equation}
\label{eq:psd}
\int_\Omega \frac{\partial^2 W}{\partial x_k \partial x_l} \rho(x) dx = 2\int_\Omega \frac{\partial \widetilde{\rho}}{\partial x_l} \frac{\partial \widetilde{\rho}}{\partial x_k} - \widetilde{\rho} \frac{\partial^2 \widetilde{\rho}}{\partial x_k \partial x_l} dx = 4 \int_\Omega \frac{\partial \widetilde{\rho}}{\partial x_l} \frac{\partial \widetilde{\rho}}{\partial x_k} dx,
\end{equation}
where the last equality holds by Green's identity with the boundary condition $\widetilde{\rho}(x) = 0$ for $x \in \partial\Omega$.
Modifying the proof from \cite{chewi2023entropic},
let us fix any non-zero vector $w \in \R^d$, and take $h \vc= \langle w, \cdot\rangle$.
Then
\[ w = \E_\mu[\nabla h] = \int_\Omega \rho(x) \nabla h(x) dx = - \int_\Omega h(x) \nabla \rho(x) dx \]
\[ \int_\Omega \nabla \rho(x) dx = 0 \]
again by Green's identity and the boundary condition $\rho(x) = 0$ for $x \in \partial \Omega$. Therefore
\[ w = - \int_\Omega (h(x) - \E_\mu[h]) \nabla \rho(x) dx = - 2\int_\Omega (h(x) - \E_\mu[h]) \widetilde{\rho}(x) \nabla \widetilde{\rho}(x) dx, \]
so that
\[ \|w\|^2 = - 2\int_\Omega (h(x) - \E_\mu[h]) \widetilde{\rho}(x) w^\top \nabla \widetilde{\rho}(x) dx.\]
Applying Cauchy-Schwarz to the right-hand side gives us
\begin{align*}
\|w\|^2 &\leq \left( \int_\Omega (h(x) - \E_\mu[h])^2 \widetilde{\rho}^2(x) dx \right)^{1/2} \left( 4 \int_\Omega (w^\top \nabla \widetilde{\rho}(x))^2 dx \right)^{1/2} \\
&= \left( \mathrm{Var}_\mu(h) \right)^{1/2} \left( w^\top \E_\mu[\nabla^2 W] w \right)^{1/2}
\end{align*}
since $\E_\mu[\nabla^2 W]$ is given by (\ref{eq:psd}).
As $\mathrm{Var}_\mu(h) = \int_\Omega (h(x) - \E_\mu[h])^2 \rho(x) dx > 0$, we have that $\E_\mu[\nabla^2 W]$ is strictly positive definite, so that
\begin{align*}
    \left\langle w, \E_\mu[\nabla^2 W]^{-1}w \right\rangle &=  - 2\int_\Omega (h(x) - \E_\mu[h]) \widetilde{\rho}(x) \left(\E_\mu[\nabla^2 W]^{-1}w \right)^\top \nabla \widetilde{\rho}(x) dx \\
    &\leq \left( \mathrm{Var}_\mu(h) \right)^{1/2} \left( 4 \int_\Omega \left(\left(\E_\mu[\nabla^2 W]^{-1}w \right)^\top \nabla \widetilde{\rho}(x)\right)^2 dx \right)^{1/2} \\
    &= \left( \mathrm{Var}_\mu(h) \right)^{1/2} \left(  \left\langle w, \E_\mu[\nabla^2 W]^{-1}w \right \rangle \right)^{1/2}
\end{align*}
by another application of Cauchy-Schwarz. (\ref{eq:cramer-rao}) follows from rearranging.
\end{proof}

\end{document}